\documentclass[11pt]{amsart}

\usepackage{amssymb}
\usepackage{mathrsfs}
\usepackage{amsfonts}
\usepackage{latexsym,amsmath,amsthm,amssymb,amsfonts}
\usepackage[usenames]{color}
\usepackage{xspace,colortbl}
\usepackage{graphicx}
\usepackage{tipa}

\newcommand{\be}{\begin{equation}}
\newcommand{\ee}{\end{equation}}
\newcommand{\beq}{\begin{eqnarray}}
\newcommand{\eeq}{\end{eqnarray}}

\newtheorem{prop}{Proposition}[section]

\newtheorem{remark}[prop]{Remark}

\def\begeq{\begin{equation}}
\def\endeq{\end{equation}}

\def\tr{{\rm tr}}

\def\odot{\setbox0=\hbox{$\bigcirc$}\relax \mathbin {\hbox
to0pt{\raise.5pt\hbox to\wd0{\hfil $\wedge$\hfil}\hss}\box0 }}

\numberwithin{equation} {section}

\numberwithin{equation}{section}
\newtheorem{theorem}{\bf Theorem}[section]
\newtheorem{proposition}[theorem]{\bf Proposition}

\newtheorem{lemma}[theorem]{\bf Lemma}

\newtheorem{corollary}[theorem]{\bf Corollary}

\allowdisplaybreaks

\begin{document}

\title[Inverse Gauss curvature flows]
{Inverse Gauss curvature flows with free boundaries in a cone}

\author{
Li Chen, Ni Xiang$^\ast$}
\address{Faculty of Mathematics and Statistics, Hubei Key Laboratory of Applied Mathematics, Hubei University,
Wuhan 430062, P.R. China} \email{chenli@hubu.edu.cn,
nixiang@hubu.edu.cn}

\thanks{Research of the second author was supported by funds from the National
Natural Science Foundation of China No.11971157. }

\thanks{$\ast$ the corresponding author}

\date{}
\begin{abstract}
We consider strictly convex hypersurfaces with the boundary which meets a strictly convex cone
perpendicularly. We prove that if these hypersurfaces expand inside this cone,
driven by $-\frac{\alpha}{n}$-th power of the Gauss curvature with $0<\alpha<1$,
then the evolution exists for all the time and
the evolving hypersurfaces converge smoothly to a piece of round
sphere after rescaling.
\end{abstract}

\maketitle {\it \small{{\bf Keywords}: Inverse Gauss curvature flow, Neumann boundary,
Monge-Amp\`ere equation.}

{{\bf MSC}: Primary 53C44, Secondary
35K96.}
}

\section{Introduction}

During the past decades, parabolic flows for hypersurfaces play an important role in differential geometry. Following the ground breaking work of Huisken
\cite{Hu}, who showed that after rescaling the surfaces
converge to round spheres, for strictly
convex initial surfaces, moving by the mean curvature flow in $\mathbb{R}^{n+1}$.
Several authors started to investigate whether the same results hold true
for surfaces moving by other curvature flows.

One of the fundamental example is powers of Gauss curvature flow in $\mathbb{R}^{n+1}$ starting with the seminal paper of Firey \cite{Firey} in 1974. The motivation of this flow arises in the study of affine geometry and of image analysis \cite{And2000}. For the short time existence, it was shown in \cite{Tso} for $\alpha=1$ and for any $\alpha>0$ in \cite{Chow}. The study of the asymptotic behavior is equivalent to the large time behavior of the normalized flow. In the special case $\alpha = \frac{1}{n}$, Chow \cite{Chow} proved convergence to a round sphere. In the affine invariant case $\alpha=\frac{1}{n+2}$, Andrews \cite{And1996} showed that the flow converged to an ellipsoid. The argument in \cite{And1996} rely crucially on the affine invariance of the equation which can not be generalized to other exponents. Convergence to spheres is known in some other special cases: for $n=1$ and $\alpha>1$ in \cite{And1998}; for $n=1$ and $\frac{1}{3}<\alpha<1$ in \cite{And2003}; for $n=2, \alpha=1$ in \cite{And}; for $n=2, \frac{1}{2}<\alpha<1$ in \cite{And2012}. However the methods in \cite{And,And2012} don't appear to work in higher dimensions since the results rely on an application of the maximum principle to a suitably chosen function of the curvature eigenvalues. As far as we know, the flow converges
to a self-similar solution for every $n\geq 2$ and $\alpha\geq \frac{1}{n+2}$. This was proved by Andrews \cite{And2000} for $\frac{1}{n+2}\leq \alpha\leq\frac{1}{n}$; and by Guan and Ni \cite{GuanNi} for $\alpha=1$; and by Andrews, Guan and Ni \cite{And2016} for $\alpha>\frac{1}{n+2}$.
For $\alpha\geq \frac{1}{n+2}$ Brendle, Choi and Daskalopoulos \cite{Bren} proved that either the limit was a round sphere, or the limit was an ellipsoid with $ \alpha=\frac{1}{n+2}$.

Compared with the above inward flows, expanding flows of the form
\begin{equation}\label{062101}
\frac{d}{dt}X = F^{-\alpha} \nu
\end{equation}
has also been studied intensively, where $F$ is a curvature function homogeneous of degree 1. The motivation to analyse the behaviour of inverse curvature flows has mostly been driven by their power to deduce geometric inequalities for hypersurfaces.
For $\alpha=1$, the equation \eqref{062101} is scale-invariant. Gerhardt \cite{Ge90} and
Urbas \cite{Ur} present a different method for proving
the convergence of star-shaped surfaces into spheres. There is a representation formula for solutions to the
inverse harmonic mean curvature flow of Smoczyk \cite{Smoczyk}. Huisken and
Ilmanen \cite{Hu01} used the inverse mean curvature flow to prove the Penrose inequality. For $0<\alpha\neq1$, the equation \eqref{062101} is non-scale-invariant. When the initial hypersurface is a sphere, Equation \eqref{062101} is
equivalent to an ODE, since the leaves of the flow will then be spheres too, and the spherical
flow will develop a blow up in finite time if $\alpha>1$.
For $0 < \alpha < 1$ Urbas \cite{Ur1991} proved that the flow exists for all time, converges to infinity and that
the rescaled hypersurfaces converge to a sphere. In the case $\alpha>1$ there are only a few special results in dimension $n = 2$ and essentially
only for the Gaussian curvature $F = K^{\frac{1}{n}}$, where $K$ is the Gaussian curvature. Schn$\ddot{u}$rer \cite{Sch2006} considered the case $\alpha = 2$ and
Li \cite{LiQ} the case $1 \leq \alpha \leq 2.$
Then Gerhardt \cite{Ge14} considered closed star-shaped initial hypersurfaces and proved that the flow existed for all time and converges to infinity, if $0 < \alpha < 1$; while in case $\alpha > 1$,
the flow blew up in finite
time, and where they assumed the initial hypersurface to be strictly convex. In both cases the
properly rescaled flows converged to the unit sphere.

The behaviour of hypersurfaces which move under various geometric flows has
been studied intensively. Although one major topic within the research
is devoted to the evolution of compact, convex hypersurfaces without
boundary, another point of interest is the behaviour of compact surfaces with boundary
which move under curvature flows. Both Dirichlet as
well as Neumann boundary conditions have been studied. Here we focus on Neumann boundary problems.
 The mean curvature flow with Neumann boundary
conditions was studied in the work of Huisken \cite{Hu89}, Stahl
\cite{Sta} and Lambert \cite{Lam}. Gauss curvature flows with
Neumann boundary condition were studied in \cite{SSch1,Sch}. Since
the Ph.D. thesis \cite{Ma2} written by Thomas Marquardt appeared,
also cf \cite{Ma1}, the inverse mean curvature flow of hypersurfaces
with boundary has been considered. Lambert and Scheuer \cite{Lam1}
obtained the flow hypersurfaces converge to the embedding of a flat
disk in the case of perpendicular to the unit sphere from the
inside. Scheuer, Wang and Xia \cite{Sch8} employed a specifically
designed locally constrained inverse harmonic mean curvature flow
with free boundary to prove Alexandrov-Fenchel inequalities for
convex hypersurfaces with free boundary in a ball. Recently, inverse
Gauss curvature flow with $\alpha=1$ expanding in a cone has been
studied by Sani \cite{San}. In details, Sani considered
hypersurfaces which were graphs over a sphere evolving in a cone,
with Neumann boundary condition and showed the flow existed for all
time and convergenced after rescaling, to a subset of a sphere.

In this paper, we study inverse Gauss curvature flow in a cone with the power
$0<\alpha<1$.
Let $\mathbb{S}^{n}$  be the sphere of radius one in
$\mathbb{R}^{n+1}$. Let $\Omega\subset\mathbb{S}^{n}$ be a portion of
$\mathbb{S}^{n}$ such that $\Sigma^{n}:=\{rx\in \mathbb{R}^{n+1}
|r>0,x \in\partial \Omega\}$ is the boundary of a smooth, strictly convex cone. We
can prove the following statement:

\begin{theorem}\label{main1.1}
Let $0<\alpha<1$ and $M_0$ be a strictly convex hypersurface
which meets $\Sigma^n$ orthogonally. That is
\begin{eqnarray*}
\partial M_0 \subset {\Sigma}^n,\quad \qquad{\langle\mu ,\nu_{0}\rangle|}_{\partial {M_0}}=0,
\end{eqnarray*}
where $\nu_{0}$ is the outward unit normal vector of $M_0$ and $\mu$ is the outward unit
normal vector of $\Sigma^n$. And assume that $$M_0=\mbox{graph}_{\mathbb{S}^n}u_{0}|_{\Omega}$$
for a positive map $u_0: \Omega\rightarrow \mathbb{R}$ with $u_0 \in C^4(\overline{\Omega})$, then we have the following conclusions:

(i) There exists a family of strictly convex hypersurfaces $M_t$ given by the unique embedding
\begin{eqnarray*}
X(\cdot, t): \Omega\rightarrow \mathbb{R}^{n+1}
\end{eqnarray*}
with $X(\partial {\Omega},t) \subset \Sigma^n$ for $t\geqslant0$,
satisfying the following system
\begin{equation}\label{Eq}
\left\{
\begin{aligned}
&\frac{d }{\partial t}X=K^{-\frac{\alpha}{n}}\nu &&in~
\Omega\times(0,\infty),\\
&\langle\mu(X),\nu(X)\rangle=0 &&on~\partial \Omega\times(0,\infty),\\
&X(\cdot,0)=M_{0}  && in~\Omega,
\end{aligned}
\right.
\end{equation}
where $K$ is the Gauss curvature of $M_t:=X(\Omega,t)$ and $\nu$ is the unit normal vector to $M_t$
pointing away from the center of the cone.

(ii) The leaves $M_t$ are graphs over $\mathbb{S}^n$,
$$M_t=\mbox{graph}_{\mathbb{S}^n}u(\cdot, t)|_{\Omega}.$$

(iii) Moreover,
the evolving hypersurfaces converge smoothly to a piece of round
sphere after rescaling.
\end{theorem}

The main technique employed in this paper is from the well-known paper
written by Lions, Trudinger and Urbas \cite{LTU}, who treated the
elliptic Neumann boundary problem for the equation of
Monge-Amp\`ere type. They took advantage of the convexity of the domain and obtained the second order
derivative estimates. Their result was extended to the
parabolic Monge-Amp\`ere type equation with Neumann boundary
problem in \cite{Sch1}. Recently, their technique was also applied to
inverse Gauss curvature flows with Neumann boundary by M. Sani \cite{San}.

\section{The corresponding scalar equation with Neumann boundary}

\subsection{Setting and General facts}

\

For convenience, we state our conventions on Riemann Curvature tensor and derivative notation.
Let $M$ be a smooth manifold and $g$ be a Riemannian metric on
$M$ with Levi-Civita connection $D$. For a $(s, r)$ tensor field $\alpha$ on $M$, its
covariant derivative $D \alpha$ is a $(s, r+1)$ tensor field given by
\begin{eqnarray*}
&&D \alpha(Y^1, .., Y^s, X_1, ..., X_r, X)
\\&=&D_{X} \alpha(Y^1, .., Y^s, X_1, ..., X_r)\\&=&X(\alpha(Y^1, .., Y^s, X_1, ..., X_r))-
\alpha(D_X Y^1, .., Y^s, X_1, ..., X_r)\\&&-...-\alpha(Y^1, .., Y^s, X_1, ..., D_X  X_r),
\end{eqnarray*}
the coordinate expression of which is denoted by
$$D \alpha=(\alpha_{k_{1}\cdot\cdot\cdot
k_{r}; k_{r+1}}^{l_{1}\cdot\cdot\cdot
l_{s}}).$$
We can continue to define the second covariant derivative of $\alpha$ as follows,
\begin{eqnarray*}
&&D^2 \alpha(Y^1, .., Y^s, X_1, ..., X_r, X, Y)
=(D_{Y}(D\alpha))(Y^1, .., Y^s, X_1, ..., X_r, X),
\end{eqnarray*}
the coordinate expression of which is denoted by
$$D^2 \alpha=(\alpha_{k_{1}\cdot\cdot\cdot
k_{r}; k_{r+1}k_{r+2}}^{l_{1}\cdot\cdot\cdot
l_{s}}).$$
In particular, for a function $u: M\rightarrow \mathbb{R}$, we have the following important identity
$$D^2 f(X, Y)=YX(f)-(D_{Y}X)f.$$
Similarly, we can also define the higher order covariant derivative of $\alpha$:
$$D^3 \alpha=D(D^2 \alpha), ... ,$$
and so on.
For simplicity, the coordinate expression of the covariant differentiation will be
denoted by indices without semicolons, e.g. $$u_{i}, \quad u_{ij} \quad \mbox{or} \quad u_{ijk}$$ for
a function $u: M\rightarrow \mathbb{R}$.

Our conventions for the Riemannian curvature (3,1)-tensor Rm is defined by
\begin{equation*}
Rm(X, Y)Z=-D_{X}D_{Y}Z+D_{Y}D_{X}Z+D_{[X,
Y]}Z.
\end{equation*}
If $X, Y \in T_{x}M$ are any two vectors spanning a subspace $P\subset T_{x}M$, the sectional curvature of $P$ is defined by
\begin{equation*}
sec(P)= \frac{\langle Rm(X, Y)X, Y\rangle}{|X|^2|Y|^2-\langle X, Y\rangle^2},
\end{equation*}
where the length of a tangent vector $X$ is defined by $|X|\doteq g(X, X)^{\frac{1}{2}}$
and $\langle X, Y\rangle\doteq g(X, Y)$.
By picking a local coordinate chart $\{x^i\}_{i=1}^{n}$ of $M$, the
component of the (3,1)-tensor $Rm$ is defined by
\begin{equation*}
Rm\bigg({\frac{\partial}{\partial x^i}}, {\frac{\partial}{\partial x^j}}\bigg){\frac{\partial}{\partial x^k}}
= R_{ijk}^{\ \ \ l}{\frac{\partial}{\partial x^l}}
\end{equation*}
and $R_{ijkl}\doteq g_{lm}R_{ijk}^{\ \ \ m}$. Then, we have the standard commutation formulas (Ricci identities):
\begin{eqnarray}\label{RI}
\alpha_{k_{1}\cdot\cdot\cdot
k_{r};\ j i}^{l_{1}\cdot\cdot\cdot
l_{s}}-\alpha_{k_{1}\cdot\cdot\cdot
k_{r};\ i j}^{l_{1}\cdot\cdot\cdot
l_{s}}=\sum_{a=1}^{r}R^{\ \ \ m}_{ijk_{a}} \alpha_{k_{1}\cdot\cdot\cdot
k_{a-1}m k_{a+1}\cdot\cdot\cdot k_{r}}^{l_{1}\cdot\cdot\cdot
l_{s}}-\sum_{b=1}^{s}R^{\ \ \ l_b}_{ijm} \alpha_{k_{1}\cdot\cdot\cdot
k_{r}}^{l_{1}\cdot\cdot\cdot
l_{b-1}m l_{b+1}\cdot\cdot\cdot l_{r}}.
\end{eqnarray}

We list some facts which will be used frequently.
For the standard sphere $\mathbb{S}^n$ with the sectional curvature $1$ and the standard metric $\sigma$,
$$R_{ijkl}=\sigma_{ik}\sigma_{jl}-\sigma_{il}\sigma_{jk}.$$
A special case of Ricci identity for a function $u: M \rightarrow \mathbb{R}$
will be
\begin{equation*}
u_{kji}-u_{kij}=R^{\ \ \ m}_{ijk} u_m.
\end{equation*}
In particular, for a function $u: \mathbb{S}^n\rightarrow \mathbb{R}$,
\begin{eqnarray}\label{3C}
u_{kji}-u_{kij}=\sigma_{ik}u_j-\sigma_{jk}u_i.
\end{eqnarray}

Let $(M, g)$ be an immersed hypersurface in
$\mathbb{R}^{n+1}$ and $\nu$ be a given unit outward normal. The
second fundamental form $h_{ij}$ of the hypersurface $M$ with respect to
$\nu$ is defined by
\begin{eqnarray*}
h_{ij}=-\left\langle\frac{\partial^2 X}{\partial x^i\partial x^j},
\nu\right\rangle_{\mathbb{R}^{n+1}}.
\end{eqnarray*}

\subsection{The scalar equation}

\

In this section, working with coordinates on the sphere,
we equivalently formulate the problem \eqref{Eq} by
the corresponding scalar equation with Neumann boundary. First, we
compute some geometric quantities induced by
the embedding of the hypersurface into $\mathbb{R}^{n+1}$.
We know that the initial hypersurface can be represented as a graph over $\Omega \subset\mathbb{S}^n$.
If the evolving hypersurface $M_{t}$ given by the embedding
\begin{eqnarray*}
X(\cdot, t): \Omega\rightarrow \mathbb{R}^{n+1}
\end{eqnarray*}
at any time $t$ may also be represented as a
graph over $\Omega\subset\mathbb{S}^n\subset \mathbb{R}^{n+1}$,
then we can make ansatz
\begin{eqnarray*}
X(x,t)=\left(u(x,t),x\right)
\end{eqnarray*}
for some function $u: \Omega \times [0,T) \rightarrow \mathbb{R}$. In this case,
we can summarize our calculations as follows.
\begin{lemma}
Let $\Omega\subset \mathbb{S}^n$, $t\geq 0$ a fixed time. Assume the embedded hypersurface $M_t$
can be represented as a graph over $\Omega\subset\mathbb{S}^n$,
$$M_t=\mbox{graph}_{\mathbb{S}^n}u(\cdot, t)|_{\Omega}.$$
Assume that $x\in \Omega$ is
described by local coordinates $\xi^{1},\ldots,\xi^{n}$, that is,
$x=x(\xi^{1},\ldots,\xi^{n})$. Let $\partial_i=\frac{\partial}{\partial\xi^i}$ be the corresponding
coordinate fields on $\Omega\subset\mathbb{S}^n$,
$\sigma_{ij}=g_{\mathbb{S}^n}(\partial_i,\partial_j)$ be the metric
on $\Omega\subset\mathbb{S}^n$ and $\partial_r=\frac{\partial}{\partial r}$ be the radial
vector field in $\mathbb{R}^{n+1}$. Let $u_{i}=D_i u$,
$u_{ij}=D^2 u(\partial_i, \partial_j)$, $u_{ijk}=D^3 u(\partial_i, \partial_j, \partial_k)$
and $u_{ijkl}=D^4 u(\partial_i, \partial_j, \partial_k, \partial_l)$ denote the
covariant derivatives of $u$ with respect to the round metric
$g_{\mathbb{S}^n}$ respectively. Then the following formulas hold.

(i) The tangential
vector on $M_{t}$ is
\begin{eqnarray*}
X_{i}=\partial_{i}+u_i\partial_{r}
\end{eqnarray*}
and the corresponding outward unit normal vector is given by
\begin{eqnarray*}
\nu=\frac{1}{v}\left(\partial_r-\frac{1}{u^2}\sigma^{ij}u_i\partial_j\right),
\end{eqnarray*}
where $v=\sqrt{1+u^{-2}|D u|^2}$
with $|D u|^2=\sigma^{ij} u_i u_j$, and $\sigma^{ij}$ is the inverse of $\sigma_{ij}$.

(ii) The induced
metric $g$ on $M_t$ has the form
\begin{equation*}
g_{ij}=u^2\sigma_{ij}+u_i u_j
\end{equation*}
and its inverse is given by
\begin{equation*}
g^{ij}=\frac{1}{u^2}\left(\sigma^{ij}-\frac{u^i u^j}{u^2 v^{2}}\right),
\end{equation*}
where we denote $u^i=\sigma^{ij}u_j$.

(iii) The second fundamental form of
$M_t$ is given by
\begin{eqnarray*}
h_{ij}=\frac{1}{v}\left(-u_{ij}+u
\sigma_{ij}+\frac{2}{u}u_iu_j\right).
\end{eqnarray*}

(iv) The Gauss curvature has the form
\begin{eqnarray*}
K=\frac{\det(h_{ij})}{\det(g_{ij})}=\frac{1}{(u^2+|Du|^2)^{\frac{n}{2}}}
\frac{\det(u^2\sigma_{ij}-u u_{ij}+2u_i u_j)}{\det(u^2\sigma_{ij}+u_i u_j)}.
\end{eqnarray*}
\end{lemma}

\begin{proof}
This formulas can be verified by direct calculation. We refer the readers to find the details in \cite{Ch16, San}.
\end{proof}

Using techniques as in Ecker \cite{Eck}, see also \cite{Ma1, Ma2}.
The problem \eqref{Eq} is reduced to solving the following scalar
equation with Neumann boundary
\begin{equation}\label{Eq-}
\left\{
\begin{aligned}
&\frac{\partial u}{\partial t}=\frac{v}{K^{\frac{\alpha}{n}}} &&in~
\Omega\times(0,\infty),\\
&D_{\mu} u=0 &&on~\partial \Omega\times(0,\infty),\\
&u(\cdot,0)=u_{0}  &&in~\Omega.
\end{aligned}
\right.
\end{equation}
Define a new function $\varphi(x,t)=\log u(x, t)$ and then
the Gauss curvature can be rewritten as
\begin{eqnarray*}\label{Gauss cur-1}
K=\frac{e^{-n\varphi}}{(1+|D\varphi|^2)^{\frac{n+2}{2}}}
\frac{\det(\sigma_{ij}-\varphi_{ij}+\varphi_i \varphi_j)}{\det(\sigma_{ij})}.
\end{eqnarray*}
The evolution equation \eqref{Eq-} can be rewrite as
\begin{eqnarray*}
\frac{\partial}{\partial t}\varphi=e^{(\alpha-1)\varphi}
(1+|D\varphi|^2)^{\frac{\beta}{n}}\frac{\det^{\frac{\alpha}{n}}(\sigma_{ij})}
{\det^{\frac{\alpha}{n}}(\sigma_{ij}-\varphi_{ij}+\varphi_i \varphi_j)}=Q(\varphi, D\varphi, D^2\varphi),
\end{eqnarray*}
where $\beta=\frac{(\alpha+1)n}{2}+\alpha.$
\begin{remark}
In particular, the power $\alpha=1$,
the equation \eqref{Eq} is scale-invariant.
In this case, the functional $Q$ in the right of the above scalar equation
does not depend on $\varphi$. However, in this paper, we consider the more complex case $\alpha\neq1$, it means that
the equation \eqref{Eq} is non- scale-invariant or the functional $Q$ depends on $\varphi$.
\end{remark}

The convexity of the initial hypersurface $M_0$ means the matrix
$$\sigma_{ij}-(\varphi_{0})_{ij}+(\varphi_{0})_{i}(\varphi_{0})_{j}$$
is positive definite up to the boundary $\partial\Omega$, where $\varphi_0=\log u_0$. Thus, the problem \eqref{Eq}
is again reduced to solving the following scalar equation with Neumann boundary
\begin{equation}\label{Evo-1}
\left\{
\begin{aligned}
&\frac{\partial \varphi}{\partial t}=Q(\varphi, D\varphi, D^{2}\varphi)
&& in ~\Omega\times(0,T),\\
&D_\mu \varphi =0 && on ~\partial \Omega\times(0,T),\\
&\varphi(\cdot,0)=\varphi_{0} && in ~ \Omega,
\end{aligned}
\right.
\end{equation}
with the matrix
$$\sigma_{ij}-(\varphi_{0})_{ij}+(\varphi_{0})_{i}(\varphi_{0})_{j}$$
is positive definite up to the boundary $\partial\Omega$.
Based on the above facts and \cite{Ma1, Ma2}, we can get the following
existence and uniqueness for the parabolic system \eqref{Eq}.
\begin{lemma}
Let $X_0$ be as in Theorem \ref{main1.1}. Then there exist some
$T>0$, a unique solution  $u \in
C^{2+\alpha,\frac{1+\alpha}{2}}(\Omega\times [0,T],\mathbb{R}^{n+1}) \cap
C^{\infty}(\Omega \times (0,T], \mathbb{R}^{n+1})$, where $\varphi(x,t)=\log
u(x,t)$, of the parabolic system \eqref{Evo-1} with the matrix
$$\sigma_{ij}-\varphi_{ij}+\varphi_{i}\varphi_{j}$$
is positive definite up to the boundary $\partial\Omega$.
Thus there exist a unique map $\psi: \Omega\times[0,T]\rightarrow \Omega$ such that
the map $\widetilde{X}$ defined by
\begin{eqnarray*}
\widetilde{X}: \Omega\times[0,T)\rightarrow \mathbb{R}^{n+1}:
(x,t)\mapsto X(\psi(x,t),t)
\end{eqnarray*}
has the same regularity as stated in Theorem \ref{main1.1} and it is
the unique solution to the parabolic system \eqref{Eq}.
\end{lemma}

Let $T^{\ast}$ be the maximal time such that there exists some $u\in
C^{2+\alpha,1+\frac{\alpha}{2}}(\Omega,[0,T^{\ast}))\cap C^{\infty}(\Omega,(0,T^{\ast}))$
which solves \eqref{Evo-1}. In the following, we shall prove a
priori estimates for those admissible solutions on $[0,T]$ where
$T<T^{\ast}$.

\section{$C^0$, $\dot{\varphi}$ and $C^1$ estimates}

\subsection{$C^0$ estimates}

To obtain $C^0$ estimates,
we need a comparison principle for parabolic equations. This follows from an interpolation
argument in \cite{San}, where a similar idea was applied in \cite{Va} to the Schouten equation.
\begin{lemma}\label{lemma3.1}
Let $\varphi$ and $\psi$ be the two solutions of \eqref{Evo-1} with
$\varphi(x,0)\leq \psi(x,0)$ for all $x\in \Omega$, then we have for $0<\alpha<1$,
\begin{equation*}
\varphi(x,t)\leq \psi(x,t)
\end{equation*}
for all $(x, t)\in \Omega\times [0, T]$.
\end{lemma}

\begin{proof}
Let $$\Lambda(x, t)=\varphi(x,t)-\psi(x,t).$$
Since $\varphi$ and $\psi$ are the solutions of scalar equation \eqref{Evo-1}, then
$$\Lambda(x, 0)\leq 0,$$
and
$$D_{\mu}\Lambda=D_{\mu}\varphi-D_{\mu}\psi=0.$$
For a real number $s \in [0,1]$, we set
\begin{eqnarray*}
\hat{v}_{ij}(x, t)[s]=\sigma_{ij}-s\varphi_{ij}+s\varphi_i \varphi_j-(1-s)\psi_{ij}
+(1-s)\psi_i \psi_j,
\end{eqnarray*}
which is clearly positive definite for any $s \in [0,1]$,
since the set of positive definite matrices is convex. Then, by the positive definite of $\hat{v}_{ij}(x, t)[s]$
we can apply the main theorem of calculus to get
\begin{equation}\label{compare1}
\begin{aligned}&\frac{\partial}{\partial t}\Lambda(x, t)=
\frac{\partial}{\partial t}\varphi(x,t)-\frac{\partial}{\partial t}\psi(x,t)\\
=&\int_{0}^{1}\frac{d}{ds}\bigg(e^{(\alpha-1)(s\varphi+(1-s)\psi)}
\frac{(1+|D(s\varphi+(1-s)\psi)|^2)^{\frac{\beta}{n}}\det^{\frac{\alpha}{n}}(\sigma_{ij})}
{\det^{\frac{\alpha}{n}}(\sigma_{ij}-s\varphi_{ij}+s\varphi_i \varphi_j-(1-s)\psi_{ij}
+(1-s)\psi_i \psi_j)}\bigg)ds.\end{aligned}
\end{equation}
Next, we do some calculations to get the derivative in the integral \eqref{compare1}.
Using the formula for the derivative of the determinant, we have
\begin{eqnarray*}
\frac{d}{d s}\det(\hat{v}_{kl})
&=&\det(\hat{v}_{kl})\hat{v}^{ij}\frac{d}{d s}\hat{v}_{ij}\\&=&
\det(\hat{v}_{kl})\hat{v}^{ij}(-\varphi_{ij}+\varphi_i \varphi_j+\psi_{ij}
-\psi_i \psi_j)\\&=&\det(\hat{v}_{kl})\hat{v}^{ij}(-\Lambda_{ij}
+\Lambda_i(\varphi_j+ \psi_j)),
\end{eqnarray*}
where $\hat{v}^{ij}$ denotes the inverse of $\hat{v}_{ij}$, which is also positive, and where the symmetry of $\hat{v}^{ij}$ yields
\begin{eqnarray*}
\hat{v}^{ij}(\varphi_i \varphi_j-\psi_i \psi_j)
=\hat{v}^{ij}(\varphi_i-\psi_i)(\varphi_j+\psi_j)=\hat{v}^{ij}\Lambda_i(\varphi_j+ \psi_j).
\end{eqnarray*}
And
\begin{eqnarray*}
&&\frac{d}{d s}(1+|D(s\varphi+(1-s)\psi)|^2)^{\frac{\beta}{n}}\\
&=&\frac{\beta}{n}(1+|D(s\varphi+(1-s)\psi)|^2)^{\frac{\beta}{n}-1}
\frac{d}{d s}|D(s\varphi+(1-s)\psi)|^2\\&=&2\frac{\beta}{n}(1+|D(s\varphi+(1-s)\psi)|^2)^{\frac{\beta}{n}-1}
\sigma^{ij}\Lambda_i(s\varphi_j+(1-s)\psi_j),
\end{eqnarray*}

\begin{eqnarray*}
\frac{d}{d s}e^{(\alpha-1)(s\varphi+(1-s)\psi)}=(\alpha-1)e^{(\alpha-1)(s\varphi+(1-s)\psi)}\Lambda.
\end{eqnarray*}
Based on these calculations, we obtain
\begin{eqnarray*}
&&\frac{d}{d s}\ln\bigg(e^{(\alpha-1)(s\varphi+(1-s)\psi)}
\frac{(1+|D(s\varphi+(1-s)\psi)|^2)^{\frac{\beta}{n}}\det^{\frac{\alpha}{n}}(\sigma_{ij})}
{\det^{\frac{\alpha}{n}}(\sigma_{ij}-s\varphi_{ij}+s\varphi_i \varphi_j-(1-s)\psi_{ij}
+(1-s)\psi_i \psi_j)}\bigg)\\&=&\frac{\alpha}{n}\bigg(\hat{v}^{ij}(\Lambda_{ij}
-\Lambda_i(\varphi_j+ \psi_j))\bigg)+2\frac{\beta}{n}(1+|D(s\varphi+(1-s)\psi)|^2)^{-1}
\sigma^{ij}\Lambda_i(s\varphi_j+(1-s)\psi_j)\\&&+(\alpha-1)\Lambda.
\end{eqnarray*}
Introducing the following notation for the positive definite
coefficient matrix of the second derivative
\begin{eqnarray*}
a^{ij}(x, t)=\frac{\alpha}{n}\hat{v}^{ij}\int_{0}^{1}e^{(\alpha-1)(s\varphi+(1-s)\psi)}
\frac{(1+|D(s\varphi+(1-s)\psi)|^2)^{\frac{\beta}{n}}\det^{\frac{\alpha}{n}}(\sigma_{kl})}
{\det^{\frac{\alpha}{n}}(\hat{v}_{kl})}d s
\end{eqnarray*}
and
\begin{eqnarray*}
b^i(x, t)&=&-a^{ij}(\varphi_j+ \psi_j)\\&&+\int_{0}^{1}e^{(\alpha-1)(s\varphi+(1-s)\psi)}
\frac{(1+|D(s\varphi+(1-s)\psi)|^2)^{\frac{\beta}{n}}\det^{\frac{\alpha}{n}}(\sigma_{kl})}
{\det^{\frac{\alpha}{n}}(\hat{v}_{kl})}d s\\&&\cdot 2\frac{\beta}{n}(1+|D(s\varphi+(1-s)\psi)|^2)^{-1}
\sigma^{ij}(s\varphi_j+(1-s)\psi_j)
\end{eqnarray*}
and
\begin{eqnarray*}
c(x, t)=(\alpha-1)\int_{0}^{1}e^{(\alpha-1)(s\varphi+(1-s)\psi)}
\frac{(1+|D(s\varphi+(1-s)\psi)|^2)^{\frac{\beta}{n}}\det^{\frac{\alpha}{n}}(\sigma_{kl})}
{\det^{\frac{\alpha}{n}}(\hat{v}_{kl})}d s.
\end{eqnarray*}
Then in view of \eqref{compare1},
\begin{equation*}
\left\{
\begin{aligned}
&\frac{\partial \Lambda}{\partial t}=a^{ij}(x, t)\Lambda_{ij}+b^{k}(x,t)\Lambda_k
+c(x, t)\Lambda=0
&& in ~\Omega\times(0,T),\\
&D_\mu \Lambda=0 && on ~\partial \Omega\times(0,T),\\
&\Lambda(\cdot,0)\leq 0 && in ~ \Omega,
\end{aligned}
\right.
\end{equation*}
here the matrix $a^{ij}(x, t)$ is positive definite and $c(x,t)\leq 0$. By using the parabolic maximum principle and Hopf'lemma,
we can conclude that $\Lambda(x, t)$
has to be nonpositive for all $t\in (0, T)$.
\end{proof}

Applying Lemma \ref{lemma3.1}, we can compare the solution of \eqref{Evo-1} with its radical solution.
\begin{lemma}
Let $\varphi$ be the solution of the parabolic system \eqref{Evo-1}, then we have,
\begin{equation}\label{C^0}
\frac{1}{1-\alpha}\ln((1-\alpha)t+e^{(1-\alpha)\varphi_1})\leq\varphi(x, t)
\leq\frac{1}{1-\alpha}\ln((1-\alpha)t+e^{(1-\alpha)\varphi_2})
\end{equation}
where $0<\alpha<1$, $\varphi_1=\inf_{\overline{\Omega}} \varphi(\cdot,0)$ and
$\varphi_2=\sup_{\overline{\Omega}}\varphi(\cdot,0)$.
\end{lemma}

\begin{proof}
Let $\varphi(x, t)=\varphi(t)$ (independent of $x$) be  the solution of \eqref{Evo-1} with $\varphi(0)=c$.
In this case, the equation \eqref{Evo-1} is reduced to an ODE
\begin{eqnarray}\label{18121901}
\frac{d}{d t}\varphi=e^{(\alpha-1)\varphi}.
\end{eqnarray}
By solving \eqref{18121901} directly,
\begin{eqnarray}\label{blow}
\varphi(t)=\frac{1}{1-\alpha}\ln((1-\alpha)t+e^{(1-\alpha)c})
\ \ \mbox{for}\ \ 0<\alpha<1.
\end{eqnarray}
Our lemma is an immediate consequence of Lemma \ref{lemma3.1}.
\end{proof}

\begin{remark}
From \eqref{blow}, we know that $\varphi(t)\rightarrow\infty$ in finite time if $\alpha>1$. Thus,
if the initial hypersurface is a sphere, the flow will blow up in finite time for $\alpha>1$.
However, similar results of
Theorem \ref{main1.1} are still expected if $\alpha>1$ as that done in \cite{Ge14}. This is will be pursued in the future paper.
\end{remark}

\begin{corollary}
If $\varphi$ satisfies \eqref{Evo-1}, then we have in the case of $0<\alpha<1$
\begin{equation*}
c_1\leq u(x, t) \Theta^{-1}(t, c)
\leq c_2, \qquad\quad \forall~ x\in
\Omega, \ t\in[0,T],
\end{equation*}
where $\Theta(t, c)=\{(1-\alpha)t+e^{(1-\alpha)c}\}^{\frac{1}{1-\alpha}}$
and $\inf_{\overline{\Omega}} \varphi(\cdot,0)\leq c\leq \sup_{\overline{\Omega}} \varphi(\cdot,0)$.
\end{corollary}

\subsection{$\dot{\varphi}$ estimates}
In this section, we shall show that $\dot{\varphi}(x, t)\Theta(t)^{1-\alpha}$ keeps
bounded during the flow. For convenience, set
\begin{eqnarray*}
w_{ij}=\sigma_{ij}-\varphi_{ij}+\varphi_i \varphi_j,
\end{eqnarray*}
$w^{ij}$ be the inverse of $w_{ij}$ and $\dot{\varphi}=\frac{\partial \varphi}{\partial t}$.
Using the evolution equation of
\eqref{Evo-1}, we obtain the following identity which will be used later,
\begin{eqnarray}\label{4.0}
\frac{\partial}{\partial t}Q&=&Q^{ij}\frac{\partial}{\partial t}\varphi_{ij}
+Q^{k}\frac{\partial}{\partial t}\varphi_k+(\alpha-1)Q\frac{\partial}{\partial t}\varphi
\nonumber\\&=&Q^{ij}\dot{\varphi}_{ij}
+Q^{k}\dot{\varphi}_k+(\alpha-1)Q\dot{\varphi},
\end{eqnarray}
where
\begin{eqnarray*}
Q^{ij}:=\frac{\partial Q}{\partial \varphi_{ij}}=\frac{\alpha}{n}\dot{\varphi}w^{ij}
\end{eqnarray*}
and
\begin{eqnarray*}
Q^k:=\frac{\partial Q}{\partial \varphi_{k}}=\frac{2\dot{\varphi}}{n}
\bigg(\frac{\beta}{1+|D\varphi|^2}\sigma^{kl}-\alpha
w^{kl}\bigg)\varphi_l.
\end{eqnarray*}
Obviously, since $w_{ij}$ positive definite, the matrix $Q^{ij}(x, t)$ is also positive definite.

\begin{lemma}\label{lemma3.2}
Let $\varphi$ be a solution of
\eqref{Evo-1}, then we have for $0<\alpha<1$
\begin{eqnarray*}
\min\{\inf_{\overline{\Omega}}\dot{\varphi}(\cdot,
0)\cdot\Theta(0)^{1-\alpha}, 1\} \leq \dot{\varphi}(x,
t)\Theta(t)^{1-\alpha}\leq
\max\{\sup_{\overline{\Omega}}\dot{\varphi}(\cdot,
0)\cdot\Theta(0)^{1-\alpha}, 1\}.
\end{eqnarray*}
\end{lemma}

\begin{proof}
Set
\begin{eqnarray*}
M(x,t)=\dot{\varphi}(x, t)\Theta(t)^{1-\alpha},
\end{eqnarray*}
then, we have by a direct computation
\begin{eqnarray*}
\frac{\partial}{\partial t}M(x,t)&=&\Theta(t)^{1-\alpha}\frac{\partial}{\partial t}Q
+(1-\alpha)\Theta(t)^{-\alpha}Q\frac{d}{dt}\Theta(t)
\\&=&\Theta(t)^{1-\alpha}\frac{\partial}{\partial t}Q
+(1-\alpha)Q
\end{eqnarray*}
in view of
\begin{eqnarray*}
\frac{d}{dt}\Theta(t)=\Theta(t)^{\alpha}.
\end{eqnarray*}
Thus, by using \eqref{4.0} we have
\begin{equation*}
\left\{
\begin{aligned}
&\frac{\partial M}{\partial
t}= Q^{ij}M_{ij}+Q^{k}M_k+(1-\alpha)\Theta^{\alpha-1}(1-M)M
&& in ~\Omega\times(0,T),\\
&D_ \mu M=0 && on ~\partial \Omega\times(0,T),\\
&M(\cdot,0)=\dot{\varphi}_0\cdot\Theta(0)^{1-\alpha}&& in ~ \Omega.
\end{aligned}
\right.
\end{equation*}
For the lower bound, on the domain $\Omega_1=\{(x, t) \in \Omega: M(x, t)<1\}$, we have
$$(1-\alpha)\Theta^{\alpha-1}(1-M(x, t))M(x, t)\geq 0.$$
 Using Hopf's Lemma, we get
\begin{eqnarray*}
M(x,t)\geq\inf_{\mathbb{S}^n}\dot{\varphi}(\cdot, 0)\cdot\Theta(0)^{1-\alpha}
\end{eqnarray*}
and
\begin{eqnarray*}
M(x,t)\geq\inf_{\partial \Omega_1}M(x, t)=1,
\end{eqnarray*}
which implies
\begin{eqnarray*}
M(x,t)\geq\min\{\inf_{\mathbb{S}^n}\dot{\varphi}(\cdot, 0)\cdot\Theta(0)^{1-\alpha}, 1\}.
\end{eqnarray*}
Similarly, we have
\begin{eqnarray*}
M(x,t)\leq\max\{\sup_{\mathbb{S}^n}\dot{\varphi}(\cdot, 0)\cdot\Theta(0)^{1-\alpha}, 1\}.
\end{eqnarray*}
Therefore, we complete the proof.
\end{proof}

\subsection{Gradient Estimates}
\begin{lemma}\label{Gradient}
Let $\varphi$ be a solution of
\eqref{Evo-1}, then we have for $0<\alpha<1$,
\begin{equation}\label{Gra-est-10}
|D\varphi|\leq \sup_{\overline{\Omega}}|D\varphi(\cdot, 0)|\cdot\frac{c^c}{[(1-\alpha)t+c]^c}, \qquad\quad \forall~ x\in
\Omega, \ t\in[0,T],
\end{equation}
where $c$ is a positive constant.
In particular, $|D\varphi|$ is bounded,
\begin{equation}\label{Gra-est}
|D\varphi|\leq \sup_{\overline{\Omega}}|D\varphi(\cdot, 0)|, \qquad\quad \forall~ x\in
\Omega, \ t\in[0,T].
\end{equation}
\end{lemma}

\begin{proof}
Set $\psi=\frac{|D\varphi|^2}{2}$. By differentiating the $\psi$,we have
\begin{equation*}
\begin{aligned}
\frac{\partial \psi}{\partial t}
&=(\frac{\partial}{\partial t}\varphi_m) \varphi^m\\
&=(\dot{\varphi})_m \varphi^m\\
&=Q_m \varphi^m,
\end{aligned}
\end{equation*}
where $Q_m$ is the covariant differentiation of $Q$.
Then,
\begin{eqnarray*}
\frac{\partial \psi}{\partial t}=Q^{ij}\varphi_{ijm} \varphi^m
+Q^k\varphi_{km} \varphi^m+(\alpha-1)Q|D \varphi|^2.
\end{eqnarray*}
Interchanging the covariant derivatives, we have
\begin{equation*}
\begin{aligned}
\psi_{ij}&=(\varphi_{mi} \varphi^m)_{j}\\&=\varphi_{mij} \varphi^m+\varphi_{mi} \varphi^{m}_{\ j}\\
&=\varphi_{imj} \varphi^m+\varphi_{mi} \varphi^{m}_{\ j}\\
&=\varphi_{ijm} \varphi^m+\sigma_{ij}|D\varphi|^2-\varphi_i\varphi_j+\varphi_{mi} \varphi^{m}_{\ j}
\end{aligned}
\end{equation*}
in view of \eqref{3C}. Thus, we have
\begin{equation}\label{grass}
\begin{aligned}
\frac{\partial \psi}{\partial t}=&Q^{ij}\psi_{ij}+Q^k\psi_k
-Q^{ij}(\sigma_{ij}|D\varphi|^2-\varphi_i \varphi_j)\\&-Q^{ij}\varphi_{mi} \varphi^{m}_{\ j}+(\alpha-1)Q|D \varphi|^2
.\end{aligned}\end{equation}
Next,  we shall consider the boundary condition. Choosing an orthonormal frame and
$e_1,e_2,\cdot\cdot\cdot,e_{n-1}\in T_x \Omega$
and $e_n=\mu$. Using Neumann boundary condition $D_ \mu \varphi =0$,
we have
\begin{equation}\label{010601}
\begin{aligned}
D_ \mu \psi&=D_{e_n}\psi=\sum_{i=1}^{n-1}D^2\varphi(e_i,e_n)D_{e_i}\varphi\\
&=\sum_{i=1}^{n-1}(D_{e_i}D_{e_n}\varphi-(D_{e_i}e_n)\varphi)D_{e_i}\varphi\\
&=-\sum_{i=1}^{n-1}((D_{e_i}e_n)\varphi) D_{e_i}\varphi\\&=
-\sum_{i=1}^{n-1}\langle D_{e_i}e_n,e_j\rangle D_{e_j}\varphi D_{e_i}\varphi\\
&=-\sum_{i=1}^{n-1}h_{ij}^{\partial \Omega}D_{e_i}\varphi D_{e_j}\varphi\\&\leq 0
,\end{aligned}\end{equation}
where $h_{ij}^{\partial \Omega}$ is the second fundamental form of $\partial\Omega$
and it is a positive definite, since $\Omega$ is convex.

Since the matrix $Q^{ij}$ and $\sigma_{ij}|D\varphi|^2-\varphi_i \varphi_j$ are
positive definite, the third and forth  terms
in the right of \eqref{grass} are non-positive. And noticing that the fifth term
in the right of \eqref{grass} can be estimated if $0<\alpha<1$ by using Lemma \ref{lemma3.2}
\begin{eqnarray*}
(\alpha-1)Q|D\varphi|^2=(\alpha-1)\Theta^{\alpha-1}Q\Theta^{1-\alpha}\psi\leq
(\alpha-1)\frac{c}{(1-\alpha)t+c}\psi.
\end{eqnarray*}
So we got the equation about $\psi$ as follows:
\begin{equation*}
\left\{
\begin{aligned}
&\frac{\partial \psi}{\partial t}\leq Q^{ij}\psi_{ij}+Q^k \psi_k-\frac{c(1-\alpha)}{(1-\alpha)t+c}\psi   &&in~
\Omega\times(0,\infty),\\
&D_\mu \psi \leq 0&&on~\partial \Omega\times(0,\infty),\\
&\psi(\;,0)=\frac{|D\varphi(\;,0)|^2}{2} &&in~\Omega.
\end{aligned}\right .\end{equation*}
Using maximum principle and Hopf'lemma, we get gradient estimates of $\varphi$.
\end{proof}

As a direct corollary, we have

\begin{corollary}
Under the assumption of Theorem \ref{main1.1}, the evolving surfaces $M_t$ are graphs over $\Omega \subset\mathbb{S}^n$.
\end{corollary}

\begin{proof}
We just need show
\begin{eqnarray*}
\langle \frac{X}{|X|}, \nu\rangle=\frac{1}{v}
\end{eqnarray*}
is bounded from above which is clearly implied by the gradient estimate \eqref{Gra-est}.
Thus, the leaves $M_t$ are star-shaped about the origin which implies that they are graphs
over $\Omega\subset\mathbb{S}^n$.
\end{proof}

Combing the gradient estimate \eqref{Gra-est} with $\dot{\varphi}$ estimate (Lemma \ref{lemma3.2}), we obtain

\begin{corollary}
If $\varphi$ satisfies \eqref{Evo-1}, then we have for $0<\alpha<1$
\begin{eqnarray}\label{w-ij}
0<c_1\leq \det(\sigma_{ij}-\varphi_{ij}+\varphi_i \varphi_j)
\leq c_2<+\infty,
\end{eqnarray}
where $c_1$ and $c_2$ are positive constants independing on $\varphi$.
\end{corollary}

\section{$C^2$ Estimates}

\

In this section, we shall obtain a priori estimates for second order derivative of $\varphi$.
\begin{theorem}\label{C^2}
Let $\varphi$ be a solution of the flow \eqref{Evo-1} and $0<\alpha<1$.
Then, there exists $C=C(n, M_0)$
such that
\begin{equation*}
|D^2\varphi(x, t)|\leq C(n, M_0),
\qquad \ \ \forall (x,
t)\in \overline{\Omega}\times [0, T^{*}).
\end{equation*}
\end{theorem}

We remark that \eqref{w-ij} together with the $C^1$-estimates \eqref{Gra-est}, implies an upper bound on
$\varphi_{ij}$. We hence only need to control $\varphi_{ij}$ from below.
Our proof will be divided into three cases.

\subsection{Interior $C^2$-estimates}

\

The main technique employed here was from M. Sani \cite{San}, and
we simplify the calculus. We consider a continuous function
$$F(\varphi)=\frac{\partial}{\partial t}\varphi-Q(\varphi, D\varphi, D^2\varphi)$$
for any $\varphi \in C^2(\overline{\Omega})$ and let
$\dot{U}=\frac{\partial U}{\partial t}$.
Then, the linearized
operator of $F$ is given by
\begin{equation*}
\begin{aligned}
\mathcal{L}_{1}U:=&\frac{d}{ds}|_{s=0}F(\varphi+sU)\\=&\dot{U}-Q^{ij}U_{ij}-
Q^{k}U_k-(\alpha-1)QU\\=&\dot{U}-\frac{\alpha}{n}\dot{\varphi}w^{ij}U_{ij}
-\frac{2\dot{\varphi}}{n}
\bigg(\frac{\beta}{1+|D\varphi|^2}\sigma^{kl}-\alpha
w^{kl}\bigg)\varphi_l U_k-(\alpha-1)QU,\end{aligned}
\end{equation*}
where
\begin{eqnarray*}
Q^{ij}:=\frac{\partial Q}{\partial \varphi_{ij}}=\frac{\alpha}{n}\dot{\varphi}w^{ij}
\end{eqnarray*}
and
\begin{eqnarray*}
Q^k:=\frac{\partial Q}{\partial \varphi_{k}}=\frac{2\dot{\varphi}}{n}
\bigg(\frac{\beta}{1+|D\varphi|^2}\sigma^{kl}-\alpha
w^{kl}\bigg)\varphi_l,
\end{eqnarray*}
these denotes are also in Section 4. Then, we define the new operator
\begin{equation*}
\begin{aligned}
\mathcal{L}U:=&\dot{U}-Q^{ij}U_{ij}-
Q^{k}U_k\\=&\dot{U}-\frac{\alpha}{n}\dot{\varphi}w^{ij}U_{ij}
-\frac{2\dot{\varphi}}{n}
\bigg(\frac{\beta}{1+|D\varphi|^2}\sigma^{kl}-\alpha
w^{kl}\bigg)\varphi_l U_k.\end{aligned}
\end{equation*}
Here, we recall the notations in Section 4:
\begin{eqnarray*}
w_{ij}=\sigma_{ij}-\varphi_{ij}+\varphi_i \varphi_j,
\end{eqnarray*}
and $w^{ij}$ is the inverse of $w_{ij}$.

First, we prove some equalities on $\mathbb{S}^n$ which will play an important role
in later computations.
\begin{lemma}
The following equalities hold on $\mathbb{S}^n$:
\begin{equation}\label{eq1}
\begin{aligned}
&w^{kl}w_{11;kl}-w^{kl}w_{kl,11}\\
=&-2\mbox{\tr}w^{kl}\varphi_{11}
+2(\mbox{\tr}w^{kl}-n+w^{kl}\varphi_{k}\varphi_{l})\\
&+2w^{kl}(\varphi_{1kl}\varphi_{1}-\varphi_{k11}\varphi_{l}).
\end{aligned}\end{equation}
\begin{equation}\label{eq2}
\begin{aligned}
&w^{kl}(w_{11;k}w_{11;l}-
w_{1k;1}w_{1l;1})\\
=&2w^{kl}w_{11;k}\varphi_lw_{11}-2
w_{11;1}\varphi_1-(w_{11})^2w^{kl}\varphi_k\varphi_l+w_{11}(\varphi_{1})^2,
\end{aligned}
\end{equation}
where the index $1$ in the identities refers to the fixed number 1.
\end{lemma}

\begin{proof}
Interchanging the covariant derivatives
\begin{eqnarray*}
\varphi_{11k}=\varphi_{1k1}+\sigma_{1k}\varphi_1-\varphi_k=\varphi_{k11}+\sigma_{1k}\varphi_1-\varphi_k
\end{eqnarray*}
in view of \eqref{3C}. Rewriting it as
\begin{eqnarray}\label{w11l}
w_{11;k}=-\varphi_{k11}-\sigma_{1k}\varphi_1+\varphi_k
+2\varphi_{1}\varphi_{1k}.
\end{eqnarray}
Since
the covariant derivatives of the curvature tensor for unit sphere vanish, we have by \eqref{RI}
\begin{equation*}\begin{aligned}
\varphi_{11kl}=&\big(\varphi_{k11}+R^{\ \ \ p}_{k11}\varphi_{p}\big)_{l}\\=&\varphi_{k11l}
+R^{\ \ \ p}_{k11}\varphi_{pl}
\\=&\varphi_{k1l1}+R^{\ \ \ p}_{l11}\varphi_{kp}+
R^{\ \ \ p}_{l1k}\varphi_{p1}+R^{\ \ \ p}_{k11}\varphi_{pl}
\\=&(\varphi_{kl1}+R^{\ \ \ p}_{l1k}\varphi_{p})_1+R^{\ \ \ p}_{l11}\varphi_{kp}+
R^{\ \ \ p}_{l1k}\varphi_{p1}+R^{\ \ \ p}_{k11}\varphi_{pl}
\\=&\varphi_{kl11}+R^{\ \ \ p}_{l11}\varphi_{kp}+
2R^{\ \ \ p}_{l1k}\varphi_{p1}+R^{\ \ \ p}_{k11}\varphi_{pl}.\end{aligned}
\end{equation*}
It follows that
\begin{equation*}\begin{aligned}
w^{kl}w_{11;kl}=&w^{kl}(-\varphi_{11kl}+(\varphi_1\varphi_1)_{kl})
\\=&w^{kl}w_{kl,11}+w^{kl}(-2R^{\ \ \ s}_{l1k}\varphi_{s1}-2R^{\ \ \ s}_{l11}\varphi_{ks}
+2\varphi_{1kl}\varphi_{1}-2\varphi_{k11}\varphi_{l})
\\ =&w^{kl}w_{kl,11}-2\mbox{\tr}w^{kl}\varphi_{11}
+2w^{kl}\varphi_{kl}+2w^{kl}(\varphi_{1kl}\varphi_{1}-\varphi_{k11}\varphi_{l})
\\ =&w^{kl}w_{kl,11}-2\mbox{\tr}w^{kl}\varphi_{11}
+2(\mbox{\tr}w^{kl}-n+w^{kl}\varphi_{k}\varphi_{l})
+2w^{kl}(\varphi_{1kl}\varphi_{1}-\varphi_{k11}\varphi_{l}).\end{aligned}
\end{equation*}
So, the equality \eqref{eq1} is obtained.

Now, we pursue the second equality. We can rewrite \eqref{w11l} as
\begin{eqnarray}\label{11k}
w_{1k;1}=w_{11;k}-w_{11}\varphi_k
+w_{1k}\varphi_1.
\end{eqnarray}
Thus,
\begin{equation*}\begin{aligned}
&w^{kl}(w_{11;k}w_{11;l}-
w_{1k;1}w_{1l;1})\\ =&w^{kl}w_{11;k}w_{11;l}\\ &-w^{kl}
(w_{11;k}-w_{11}\varphi_k
+w_{1k}\varphi_1)(w_{11;l}-w_{11}\varphi_l
+w_{1l}\varphi_1)\\=&-2w^{kl}
w_{11;k}(-w_{11}\varphi_l
+w_{1l}\varphi_1)\\&-w^{kl}
(-w_{11}\varphi_k
+w_{1k}\varphi_1)(-w_{11}\varphi_l
+w_{1l}\varphi_1)\\ =&2w^{kl}w_{11;k}\varphi_lw_{11}-2
w_{11;1}\varphi_1-(w_{11})^2w^{kl}\varphi_k\varphi_l+w_{11}(\varphi_{1})^2.\end{aligned}
\end{equation*}
\end{proof}

\begin{remark}
Although the equality \eqref{eq2} is also obtained in \cite{San}, the great difference between ours and the one in \cite{San} is rewriting \eqref{w11l} as another form \eqref{11k}. This improvement simplifies the calculation in our paper.

\end{remark}

To proceeding our proof, we need the evolution.
\begin{lemma}
Under the flow \eqref{Evo-1}, the following evolution equations hold
true
\begin{equation}\label{2-Ev}
\begin{aligned}
\mathcal{L}(\frac{1}{2}|D\varphi|^2)=&
-\frac{\alpha\dot{\varphi}}{n}\bigg((1+|D\varphi|^2)w^{ij}\sigma_{ij}+
(1+|D\varphi|^2)w^{ij}\varphi_{i}\varphi_j\\
&-\Delta\varphi-|D\varphi|^2-n\bigg)
+(\alpha-1)\dot{\varphi}|D\varphi|^2.\end{aligned}
\end{equation}

\begin{equation}\label{w11}\begin{aligned}
\mathcal{L}w_{11}=&-\frac{(\dot{\varphi}_1)^2}{\dot{\varphi}}+
\frac{\alpha}{n}\dot{\varphi} w^{kl}_{\ \
;1}w_{kl;1}\\
&+\frac{4\beta\dot{\varphi}}{n}
\frac{1}{(1+|D\varphi|^2)^2}(\sigma^{kl}\varphi_{k}\varphi_{l1})^2-\frac{2\beta\dot{\varphi}}{n}
\frac{1}{1+|D\varphi|^2}\sigma^{kl}\varphi_{k1}\varphi_{l1}
\\ &+\frac{2\dot{\varphi}}{n}
\frac{\beta}{1+|D\varphi|^2}\bigg((\varphi_1)^2-|D\varphi|^2\bigg)
+\frac{2\alpha}{n}\dot{\varphi}
(-w_{11}\mbox{tr}w^{ij}+n)\\ &+(\alpha-1)\dot{\varphi}(w_{11}-(\varphi_1)^2-1)
+2(\alpha-1)\dot{\varphi}\varphi_1.\end{aligned}
\end{equation}

\begin{eqnarray}\label{w11-}
\mathcal{L}(\gamma^i \varphi_i)&=&\frac{\alpha}{n} \dot{\varphi} w^{ij}\bigg(\sigma_{il}\varphi_j\gamma^{l}
-\sigma_{ij}\varphi_l \gamma^{l}-2\varphi_{ki}\gamma^{k}_{\ ;j}
-\varphi_{k}\gamma^{k}_{\ ; ij}+2\varphi_j \varphi_k \gamma^{k}_{\ ; i}\bigg)\\ \nonumber&&-\frac{2\beta}{n} \dot{\varphi}
\frac{1}{1+|D \varphi|^2}\varphi^k\varphi_{l} \gamma^{l}_{\ ;k}
+(\alpha-1)\dot{\varphi}\varphi_l\gamma^{l},
\end{eqnarray}
where $\Delta$ is the Laplace of $D$ and $\gamma^i:\overline{\Omega}\rightarrow\mathbb{R}$ is a smooth
function that does not depend on $\varphi$.
\end{lemma}

\begin{proof}
We begin to prove the first evolution equation. Clearly,
\begin{eqnarray*}
\mathcal{L}(\frac{1}{2}|D\varphi|^2)&=&\sigma^{rs}
\dot{\varphi}_r\varphi_s-Q^{ij}(\frac{1}{2}|D\varphi|^2)_{ij}-Q^k
(\frac{1}{2}|D\varphi|^2)_k.
\end{eqnarray*}
Using the evolution equation \eqref{Evo-1},
the first term on the right of the above equation becomes
\begin{eqnarray*}
\sigma^{rs} \dot{\varphi}_r\varphi_s=\sigma^{rs}
(Q^{ij}\varphi_{ijr}+Q^k\varphi_{kr}+(\alpha-1)\dot{\varphi}
\varphi_r)\varphi_s.
\end{eqnarray*}
We have in view of \eqref{3C}
\begin{equation*}\begin{aligned}
\varphi_{ijr}\varphi^r=&\varphi_{irj}\varphi^r+R^{\ \ \ m}_{rji}\varphi_{m}\varphi^r=\varphi_{rij}\varphi^r+
(\delta^{m}_{j}\sigma_{ir}-\delta^{m}_{r}\sigma_{ij})\varphi_{m}\varphi^r\\
=&(\frac{1}{2}|D\varphi|^2)_{ij}-\sigma^{rs}\varphi_{ri}\varphi_{sj}
+\varphi_i\varphi_j-|D\varphi|^2\sigma_{ij}.\end{aligned}
\end{equation*}
Thus, we obtain
\begin{eqnarray*}
\mathcal{L}(\frac{1}{2}|D\varphi|^2)=-
Q^{ij}\varphi_{ri}\varphi_{sj}\sigma^{rs}
+Q^{ij}(\varphi_i\varphi_j-|D\varphi|^2\sigma_{ij})+(\alpha-1)\dot{\varphi}|D\varphi|^2.
\end{eqnarray*}
It follows that
\begin{equation*}\begin{aligned}
\mathcal{L}(\frac{1}{2}|D\varphi|^2)=&
-\frac{\alpha\dot{\varphi}}{n}\bigg((1+|D\varphi|^2)w^{ij}\sigma_{ij}+
(1+|D\varphi|^2)w^{ij}\varphi_{i}\varphi_j\\&-\Delta\varphi-|D\varphi|^2-n\bigg)
+(\alpha-1)\dot{\varphi}|D\varphi|^2\end{aligned}
\end{equation*}
in view of
\begin{eqnarray*}
w^{ij}\varphi_{ri}=w^{ij}(\sigma_{ri}-w_{ri}+\varphi_r\varphi_i)
=w^{ij}\sigma_{ri}-\delta^{j}_{r}+w^{ij}\varphi_r\varphi_i.
\end{eqnarray*}
Now, we prove the second evolution equation. Clearly,
\begin{eqnarray*}
\mathcal{L}(w_{11})=\dot{w}_{11}-Q^{ij}w_{11;ij}-Q^k w_{11;k}.
\end{eqnarray*}
Using the evolution equation \eqref{Evo-1}, we have
\begin{equation*}\begin{aligned}
\dot{w}_{11}=&-\dot{\varphi}_{11}+2\dot{\varphi}_{1}\varphi_1\\=&
-\bigg(-\frac{\alpha\dot{\varphi}}{n}w^{kl}w_{kl;1}
+\frac{\dot{\varphi}}{n}\frac{2\beta}{1+|D\varphi|^2}\sigma^{kl}\varphi_k\varphi_{l1}+(\alpha-1)\dot{\varphi}
\varphi_1\bigg)_1+2\dot{\varphi}_{1}\varphi_1
\\=&-\frac{(\dot{\varphi}_1)^2}{\dot{\varphi}}+
\frac{\alpha\dot{\varphi}}{n} w^{kl}_{\ \
;1}w_{kl;1}+\frac{\alpha\dot{\varphi}}{n}w^{kl}w_{kl;11}+\frac{4\beta\dot{\varphi}}{n}
\frac{1}{(1+|D\varphi|^2)^2}(\sigma^{kl}\varphi_{k}\varphi_{l1})^2
\\&-\frac{2\beta\dot{\varphi}}{n}
\frac{1}{1+|D\varphi|^2}\sigma^{kl}(\varphi_{k1}\varphi_{l1}
+\varphi_{k}\varphi_{l11})
-(\alpha-1)\dot{\varphi}\varphi_{11}+2\dot{\varphi}_1\varphi_1.\end{aligned}
\end{equation*}
Inserting \eqref{eq1} into the above equality, we obtain
\begin{equation}\label{w11-1}\begin{aligned}
\dot{w}_{11}=&-\frac{(\dot{\varphi}_1)^2}{\dot{\varphi}}+
\frac{\alpha}{n}\dot{\varphi} w^{kl}_{\ \
;1}w_{kl;1}+\frac{\alpha}{n}\dot{\varphi}
w^{kl}w_{11;kl}\\ &+\frac{2\alpha}{n}\dot{\varphi}\mbox{tr}w^{ij}\varphi_{11}
-\frac{2\alpha}{n}\dot{\varphi}(\mbox{tr}w^{ij}-n+w^{ij}\varphi_{i}\varphi_{j})\\
&+\frac{4\beta\dot{\varphi}}{n}
\frac{1}{(1+|D\varphi|^2)^2}(\sigma^{kl}\varphi_{k}\varphi_{l1})^2-\frac{2\beta\dot{\varphi}}{n}
\frac{1}{1+|D\varphi|^2}\sigma^{kl}\varphi_{k1}\varphi_{l1}
\\&-\frac{2\dot{\varphi}}{n}(
\frac{\beta}{1+|D\varphi|^2}\sigma^{kl} -\alpha w^{kl})\varphi_l
\varphi_{k11}+2(\dot{\varphi}_1-\frac{\alpha}{n}\dot{\varphi}w^{kl}\varphi_{1kl})\varphi_1
-(\alpha-1)\dot{\varphi}\varphi_{11}.\end{aligned}
\end{equation}
Since
\begin{equation*}\begin{aligned}
-2w^{kl}\varphi_{1kl}\varphi_{1}
=&-2w^{kl}(\varphi_{kl1}-\sigma_{k1}\varphi_l
+\sigma_{kl}\varphi_1)\varphi_{1}\\=&2w^{kl}w_{kl;1}\varphi_1
-2w^{kl}(\varphi_k\varphi_l)_1\varphi_1+2w^{kl}\sigma_{k1}\varphi_l\varphi_1
-2w^{kl}\sigma_{kl}(\varphi_1)^2\\=&2w^{kl}w_{kl;1}\varphi_1
-4w^{kl}\varphi_k\varphi_{l1}\varphi_1+2w^{kl}\sigma_{k1}\varphi_l\varphi_1
-2w^{kl}\sigma_{kl}(\varphi_1)^2,\end{aligned}
\end{equation*}
the second term in the last line of \eqref{w11-1} can be rewritten as
\begin{equation*}\begin{aligned}
&2(\dot{\varphi}_1-\frac{\alpha}{n}\dot{\varphi}w^{kl}\varphi_{1kl})\varphi_1
\\=&2(-\frac{\alpha\dot{\varphi}}{n}w^{kl}w_{kl;1}
+\frac{2\dot{\varphi}}{n}\frac{\beta}{1+|D\varphi|^2}\sigma^{kl}\varphi_k\varphi_{l1}
+(\alpha-1)\dot{\varphi}\varphi_1-\frac{\alpha}{n}\dot{\varphi}w^{kl}\varphi_{1kl})\varphi_1
\\=&\frac{4\dot{\varphi}}{n}\frac{\beta}{1+|D\varphi|^2}\sigma^{kl}
\varphi_k\varphi_{l1}\varphi_1+\frac{2\alpha}{n}\dot{\varphi}w^{kl}
(-2\varphi_k\varphi_{l1}\varphi_1+\sigma_{k1}\varphi_l\varphi_1
-\sigma_{kl}(\varphi_1)^2)\\&+2(\alpha-1)\dot{\varphi}\varphi_1.\end{aligned}
\end{equation*}
And the first term in the last line of \eqref{w11-1} can be rewritten as
\begin{equation*}\begin{aligned}
&-\frac{2\dot{\varphi}}{n}( \frac{\beta}{1+|D\varphi|^2}\sigma^{kl}
-\alpha w^{kl})\varphi_l
\varphi_{k11}\\=&-\frac{2\dot{\varphi}}{n}(
\frac{\beta}{1+|D\varphi|^2}\sigma^{kl} -\alpha w^{kl})\varphi_l
(-w_{11;k}-\sigma_{1k}\varphi_1+\varphi_k+2\varphi_1\varphi_{1k}),\end{aligned}
\end{equation*}
in view of \eqref{w11l}, it follows that
\begin{equation*}\begin{aligned}
&-\frac{2\dot{\varphi}}{n}( \frac{\beta}{1+|D\varphi|^2}\sigma^{kl}
-\alpha w^{kl})\varphi_l \varphi_{k11}-Q^k
w_{11;k}\\=&-\frac{2\dot{\varphi}}{n}(
\frac{\beta}{1+|D\varphi|^2}\sigma^{kl} -\alpha w^{kl})\varphi_l
(-\sigma_{1k}\varphi_1+\varphi_k+2\varphi_1\varphi_{1k}).\end{aligned}
\end{equation*}
Therefore,
\begin{equation*}\label{w11-2}\begin{aligned}
\mathcal{L}w_{11}=&-\frac{(\dot{\varphi}_1)^2}{\dot{\varphi}}+
\frac{\alpha}{n}\dot{\varphi} w^{kl}_{\ \
;1}w_{kl;1}+\frac{2\alpha}{n}\dot{\varphi}\mbox{tr}w^{ij}\varphi_{11}
-\frac{2\alpha}{n}\dot{\varphi}(\mbox{tr}w^{ij}-n+w^{ij}\varphi_{i}\varphi_{j})\\
&+\frac{4\beta\dot{\varphi}}{n}
\frac{1}{(1+|D\varphi|^2)^2}(\sigma^{kl}\varphi_{k}\varphi_{l1})^2-\frac{2\beta\dot{\varphi}}{n}
\frac{1}{1+|D\varphi|^2}\sigma^{kl}\varphi_{k1}\varphi_{l1}
\\ &+\frac{2\dot{\varphi}}{n}
\frac{\beta}{1+|D\varphi|^2}\bigg((\varphi_1)^2-|D\varphi|^2\bigg)
+\frac{2\alpha}{n}\dot{\varphi}w^{kl}
\bigg(-\sigma_{kl}(\varphi_1)^2+\varphi_k\varphi_l\bigg)
\\ &-(\alpha-1)\dot{\varphi}\varphi_{11}+2(\alpha-1)\dot{\varphi}\varphi_1,\end{aligned}
\end{equation*}
which implies that
\begin{equation}\label{w11-3}\begin{aligned}
\mathcal{L}w_{11}=&-\frac{(\dot{\varphi}_1)^2}{\dot{\varphi}}
+\frac{\alpha}{n}\dot{\varphi} w^{kl}_{\ \
;1}w_{kl;1}\\
&+\frac{4\beta\dot{\varphi}}{n}
\frac{1}{(1+|D\varphi|^2)^2}(\sigma^{kl}\varphi_{k}\varphi_{l1})^2-\frac{2\beta\dot{\varphi}}{n}
\frac{1}{1+|D\varphi|^2}\sigma^{kl}\varphi_{k1}\varphi_{l1}
\\ &+\frac{2\dot{\varphi}}{n}
\frac{\beta}{1+|D\varphi|^2}\bigg((\varphi_1)^2-|D\varphi|^2\bigg)
\\ &+\frac{2\alpha}{n}\dot{\varphi}
(-w_{11}\mbox{tr}w^{ij}+n)+(\alpha-1)\dot{\varphi}(w_{11}-(\varphi_1)^2-1)+2(\alpha-1)\dot{\varphi}\varphi_1.\end{aligned}
\end{equation}
Now, we only leave the third equality to prove.
Differentiating the function $\gamma^{k}\varphi_{k}$ twice with $x$, we have
\begin{eqnarray*}
(\varphi_{k}\gamma^{k})_{i}=\varphi_{ki}\gamma^{k}+\varphi_{k}\gamma^{k}_{\ ;i}
\end{eqnarray*}
and
\begin{eqnarray*}
(\varphi_{k}\gamma^{k})_{ij}=\varphi_{kij}\gamma^{k}+\varphi_{ki}\gamma^{k}_{\ ;j}+\varphi_{kj}\gamma^{k}_{\ ;i}
+\varphi_{k}\gamma^{k}_{\ ;ij}.
\end{eqnarray*}
Differentiating $\varphi_{k}\gamma^{k}$  with $t$, we have
\begin{equation*}\begin{aligned}
(\varphi_{k}\gamma^{k})_{t}=&\varphi_{kt}\cdot\gamma^k\\=&\bigg(Q^{ij}\varphi_{ijl}+Q^{k}\varphi_{kl}
+(\alpha-1)Q\varphi_l\bigg)\gamma^{l}\\
=&\bigg(Q^{ij}\varphi_{lij}+Q^{ij}\varphi_{j}\sigma_{il}
-Q^{ij}\varphi_{l}\sigma_{ij}+Q^{k}\varphi_{kl}+(\alpha-1)Q\varphi_l\bigg)\gamma^{l}.\end{aligned}
\end{equation*}
Therefore,
\begin{equation*}\begin{aligned}
\mathcal{L}(\varphi_{k}\gamma^{k})=&(\varphi_{k}\gamma^{k})_t-Q^{ij}
(\varphi_{k}\gamma^{k})_{ij}-Q^k(\varphi_{i}\gamma^{i})_{k}\\=&Q^{ij}\bigg(\sigma_{il}\varphi_j\gamma^{l}
-\sigma_{ij}\varphi_l \gamma^{l}-2\varphi_{ki}\gamma^{k}_{\ ;j}-\varphi_k \gamma^{k}_{\ ;ij}\bigg)-Q^{k}
\varphi_{l} \gamma^{l}_{\ ;k}\\&
+(\alpha-1)Q\varphi_l\gamma^{l}.\end{aligned}
\end{equation*}
Then, we obtain our result by inserting $Q^{ij}$ and $Q^k$ into the above equality.
\end{proof}

Let $\mu$ be a smooth extension of the outward unit normal to $\partial\Omega$ that vanishes outside
a tubular neighborhood of $\partial\Omega$. We define for
$(x, \xi_{1}, \xi_{2}, t)\in \overline{\Omega}\times \mathbb{R}^{n}\times \mathbb{R}^{n}\times [0, T]$
\begin{eqnarray*}
w^{\prime}(x, \xi_1, \xi_2, t)
=-\mu^{i}_{\ ;j}\varphi_{i}(\langle\xi_1, \mu\rangle\xi^{\prime j}_2+\langle\xi_2, \mu\rangle\xi^{\prime j}_1),
\end{eqnarray*}
where $$\xi^{\prime}_i=\xi_i-\langle\xi_i, \mu\rangle\mu$$
indicate the tangential component of the vector $\xi_i$, with $i=1,2$ and where $\langle,\rangle$
is the inner product induced by $\sigma$. Moreover, let
$w^{\prime}_{ij}(x, t): \overline{\Omega}\times [0, T]\rightarrow \mathbb{R}^n$, with $1\leq i,j\leq n$,
represent the component functions
\begin{eqnarray*}
w^{\prime}_{ij}(x, t)=-\mu^{q}_{\ ; p}\varphi_q[\sigma_{ki}
\mu^k(\delta^{p}_{j}-\sigma_{lj}\mu^l\mu^p)+\sigma_{kj}
\mu^k(\delta^{p}_{i}-\sigma_{li}\mu^l\mu^p)],
\end{eqnarray*}
of the symmetric 2-tensor filed $w^{\prime}$.

\begin{remark}
$w^{\prime}(x, \xi_1, \xi_2, t)$ is not an important part in the following interior estimate,
but will paly a great role in the later non-tangential and non-normal boundary estimate.
\end{remark}

We define for $(x, \xi, t)\in \overline{\Omega}\times \mathbb{R}^{n}\times [0, T]$
as that done in \cite{San}
\begin{eqnarray*}
W(x, \xi, t)=\log \bigg(\frac{[w_{ij}(x,t)+w^{\prime}_{ij}(x,t)]
\xi^i \xi^j}{\sigma_{ij}\xi^i \xi^j}+C\bigg)+\frac{1}{2}\lambda |D\varphi|^2,
\end{eqnarray*}
where $C$ and $\lambda$ are constants which will be chosen later.

\begin{proposition}\label{interior}
Let $\varphi$ be a solution of the flow \eqref{Evo-1}, assume $W$
attains its maximum in $\Omega\times \mathbb{S}^{n-1}\times [0, T]$ for some fixed $T<T^{*}$.
Then, there exists $C=C(n, M_0)$
such that
\begin{equation*}
C(n, M_0)\leq \varphi_{ij}\xi^i\xi^j, \qquad \ \ \forall (x, \xi,
t)\in \overline{\Omega}\times \mathbb{S}^{n-1}\times [0, T].
\end{equation*}
\end{proposition}

\begin{proof}
Assume $W(x, \xi, t)$ achieves its maximum at
$(x_0, \xi_0, t_0)\in \Omega\times \mathbb{S}^{n-1}\times [0, T]$.
Choose Riemannian normal coordinates at $x_0$ such that at
this point we have
\begin{eqnarray*}
\sigma_{ij}(x_0)=\delta_{ij}, \ \ \ \partial_{k}\sigma_{ij}(x_0)=0.
\end{eqnarray*}
And we further rotate the coordinate system at $(x_0, t_0)$
such that the matrix $w_{ij}+w^{\prime}_{ij}$ is diagonal, i.e.
\begin{eqnarray*}
w_{ij}+w^{\prime}_{ij}=(w_{ii}+w^{\prime}_{ii})\delta_{ij}
\end{eqnarray*}
with
\begin{eqnarray*}
w_{nn}+w^{\prime}_{nn}\leq\cdot\cdot\cdot\leq w_{22}+w^{\prime}_{22}
\leq w_{11}+w^{\prime}_{11}.
\end{eqnarray*}
Thus, since the matrix $w_{ij}$ is positive definite, we have at $(x_0, t_0)$
\begin{eqnarray}\label{ij}
|w_{ii}|\leq w_{11}+c \ \ \mbox{and} \ \ |w_{ij}|\leq c \ \
\mbox{for} \ \ i\neq j
\end{eqnarray}
in view of the $C^1$-estimate \eqref{Gra-est}.
Set $\xi_1(x)=(1,0,...,0)$ around a neighbor of $x_0$. Clearly, $\xi_1(x_0)=\xi_0$ and  there
holds at $(x_0, t_0)$
\begin{eqnarray*}
w_{11}+w^{\prime}_{11}=\sup_{\xi\in \mathbb{S}^n}\frac{[w_{ij}(x,t)+w^{\prime}_{ij}(x,t)]
\xi^i \xi^j}{\sigma_{ij}\xi^i \xi^j}
\end{eqnarray*}
and in a neighborhood of $(x_0, t_0)$
\begin{eqnarray*}
w_{11}+w^{\prime}_{11}\leq\sup_{\xi\in \mathbb{S}^n}\frac{[w_{ij}(x,t)+w^{\prime}_{ij}(x,t)]
\xi^i \xi^j}{\sigma_{ij}\xi^i \xi^j}.
\end{eqnarray*}
Furthermore, it is easy to check that the covariant(at least up to the second order) and the first time derivatives of
\begin{eqnarray*}
\frac{[w_{ij}(x,t)+w^{\prime}_{ij}(x,t)]
\xi_{1}^i \xi_{1}^j}{\sigma_{ij}\xi_{1}^i \xi_{1}^j}
\end{eqnarray*}
and
\begin{eqnarray*}
w_{11}+w^{\prime}_{11}
\end{eqnarray*}
do coincide at $(x_0, t_0)$(in normal coordinate).
Without loss of generality, we treat $w_{11}+w^{\prime}_{11}$ like a
scalar and pretend that $W$ is defined by
\begin{eqnarray*}
W(x, t)=\log(w_{11}+w^{\prime}_{11}+C)+\frac{1}{2}\lambda|D\varphi|^2,
\end{eqnarray*}
which achieves its maximum at $(x_0, t_0)\in \Omega\times[0, T]$.
Here, noticing that we can choose $C$ large enough satisfying
\begin{eqnarray}\label{C0}
0 \leq w^{\prime}_{11}+C,
\end{eqnarray}
since $w^{\prime}_{11}$ is bounded by the $C^1$-estimate
\eqref{Gra-est}.

In the following, we want to compute
\begin{equation*}\begin{aligned}
\mathcal{L}W=&\dot{W}-Q^{ij}W_{ij}- Q^{k}W_k\\=&\mathcal{L}(\log(w_{11}+w^{\prime}_{11}+C))
+\frac{1}{2}\lambda\mathcal{L}(|D\varphi|^2).\end{aligned}
\end{equation*}
First, after a simple calculation, we can rewrite the first term in
the following form
\begin{equation*}\begin{aligned}
&\mathcal{L}(\log(w_{11}+w^{\prime}_{11}+C)\\=&
\frac{\mathcal{L}w_{11}}{w_{11}+w^{\prime}_{11}+C}
+\frac{\mathcal{L}w^{\prime}_{11}}{w_{11}+w^{\prime}_{11}+C}
+\frac{\alpha}{n}\dot{\varphi}w^{ij}
\frac{(w_{11;i}+w^{\prime}_{11;i})(w_{11;j}+w^{\prime}_{11;j})}
{(w_{11}+w^{\prime}_{11}+C)^2}.\end{aligned}
\end{equation*}
Now, we begin to estimate $\mathcal{L}w_{11}$ through the evolution
\eqref{w11}. Using the Cauchy-Schwarz inequality, the second line of
\eqref{w11} takes the form
\begin{equation}\label{2-line}\begin{aligned}
&\frac{2\beta\dot{\varphi}}{n}
\frac{1}{1+|D\varphi|^2}\bigg(\frac{2}{1+|D\varphi|^2}(\sigma^{kl}\varphi_{k}\varphi_{l1})^2
-\sigma^{kl}\varphi_{k1}\varphi_{l1}\bigg)\\ \leq&\frac{2\beta\dot{\varphi}}{n}
\frac{1}{1+|D\varphi|^2}\frac{|D\varphi|^2-1}{|D\varphi|^2+1}
\sigma^{kl}\varphi_{k1}\varphi_{l1}\\ \leq&\frac{2\beta\dot{\varphi}}{n}
\frac{1}{1+|D\varphi|^2}
\sigma^{kl}\varphi_{k1}\varphi_{l1}.\end{aligned}
\end{equation}
On the other hand,
\begin{eqnarray*}
\sigma^{kl}\varphi_{k1}\varphi_{l1}=\sigma^{kl}w_{k1}w_{l1}+\sigma_{11}-2w_{11}
+2(\varphi_1)^2-2\varphi_1\sigma^{kl}w_{k1}\varphi_l+(\varphi_1)^2|D\varphi|^2.
\end{eqnarray*}
Using \eqref{ij}, together with the $C^1$-estimate \eqref{Gra-est},
\eqref{2-line} becomes
\begin{equation}\label{2-line-1}\begin{aligned}
&\frac{2\beta\dot{\varphi}}{n}
\frac{1}{1+|D\varphi|^2}\bigg(\frac{2}{1+|D\varphi|^2}(\sigma^{kl}\varphi_{k}\varphi_{l1})^2
-\sigma^{kl}\varphi_{k1}\varphi_{l1}\bigg)\\ \leq&\frac{2\beta\dot{\varphi}}{n}
\frac{1}{1+|D\varphi|^2}
(\sigma^{kl}w_{k1}w_{l1}+cw_{11}+c).\end{aligned}
\end{equation}
Inserting \eqref{2-line-1} into $\mathcal{L}w_{11}$,
abandoning the non-positive terms and using the $C^1$-estimate
\eqref{Gra-est} again, we obtain
\begin{eqnarray}\label{2-line-2}
\mathcal{L}w_{11}\leq\frac{2\beta\dot{\varphi}}{n}
\frac{1}{1+|D\varphi|^2}
\sigma^{kl}w_{k1}w_{l1}
+c\ \dot{\varphi}
(\mbox{tr}w^{ij}+w_{11}+1)+
\frac{\alpha}{n}\dot{\varphi} w^{kl}_{\ \
;1}w_{kl;1}.
\end{eqnarray}
Next, recalling  \eqref{2-Ev} and using the $C^1$-estimate
\eqref{Gra-est},
\begin{eqnarray}\label{2-Ev-1}
&&\lambda\mathcal{L}(\frac{1}{2}|D\varphi|^2)\\ \nonumber&=&
-\frac{\alpha\lambda\dot{\varphi}}{n}\bigg((1+|D\varphi|^2)w^{ij}\sigma_{ij}+
(1+|D\varphi|^2)w^{ij}\varphi_{i}\varphi_j+w_{ij}\sigma^{ij}
-2|D\varphi|^2-2n\bigg)\\ \nonumber
&&+(\alpha-1)\lambda\dot{\varphi}|D\varphi|^2\\&\leq&
-\frac{\alpha\lambda\dot{\varphi}}{n}\bigg(w_{ij}\sigma^{ij}+
(1+|D\varphi|^2)w^{ij}\varphi_{i}\varphi_j-2|D\varphi|^2\bigg)\nonumber\\ \nonumber
&&-\frac{\alpha\lambda\dot{\varphi}}{n}\mbox{tr}w^{ij}+c\lambda\dot{\varphi}.
\end{eqnarray}
Then, it follows in view of \eqref{2-line-2} and \eqref{2-Ev-1}
\begin{equation*}\begin{aligned}
\mathcal{L}W\leq&
\frac{1}{w_{11}+w^{\prime}_{11}+C}\bigg(\frac{2\beta\dot{\varphi}}{n}
\frac{1}{1+|D\varphi|^2}
\sigma^{kl}w_{k1}w_{l1}
+c \ \dot{\varphi}
(\mbox{tr}w^{ij}+w_{11}+1)+
\frac{\alpha}{n}\dot{\varphi} w^{kl}_{\ \
;1}w_{kl;1}\bigg)
\\&+\frac{\mathcal{L}w^{\prime}_{11}}{w_{11}+w^{\prime}_{11}+C}
+\frac{\alpha}{n}\dot{\varphi}w^{ij}
\frac{(w_{11;i}+w^{\prime}_{11;i})(w_{11;j}+w^{\prime}_{11;j})}
{(w_{11}+w^{\prime}_{11})^2}\\&-\frac{\alpha\lambda\dot{\varphi}}{n}\bigg(w_{ij}\sigma^{ij}+
(1+|D\varphi|^2)w^{ij}\varphi_{i}\varphi_j-|D\varphi|^2\bigg)\\
&-\frac{\alpha\lambda\dot{\varphi}}{n}\mbox{tr}w^{ij}+c\lambda\dot{\varphi}.\end{aligned}
\end{equation*}
To make progress, we need to estimate
\begin{equation*}\begin{aligned}
&\frac{1}{w_{11}+w^{\prime}_{11}+C}\frac{2\beta\dot{\varphi}}{n}
\frac{1}{1+|D\varphi|^2}
\sigma^{kl}w_{k1}w_{l1}
-\frac{\alpha\lambda\dot{\varphi}}{n}w_{ij}\sigma^{ij}
\\ \leq &c\ \dot{\varphi}(\frac{(w_{11}+c)^2}{w_{11}+w^{\prime}_{11}+C}-\lambda w_{11})\\ \leq &
c(1-\lambda)\dot{\varphi}w_{11}\end{aligned}
\end{equation*}
in view of \eqref{ij} and \eqref{C0}, where we assume that $w_{11}\geq 1$, otherwise $w_{11}$ is upper bounded
and our theorem holds true.
Now, we only leave the term
$\mathcal{L}w^{\prime}_{11}$ to estimate. Clearly, $w^{\prime}_{11}$
can be rewritten as
\begin{eqnarray*}
w^{\prime}_{11}=\gamma^i\varphi_i+C
\end{eqnarray*}
with $\gamma^i:\overline{\Omega}\rightarrow\mathbb{R}$
that does not depend on $\varphi$. Recalling \eqref{w11-}, we have
\begin{eqnarray*}
\mathcal{L}(\gamma^i \varphi_i)&=&\frac{\alpha}{n} \dot{\varphi} w^{ij}\bigg(\sigma_{il}\varphi_j\gamma^{l}
-\sigma_{ij}\varphi_l \gamma^{l}-2\varphi_{ki}\gamma^{k}_{\ ;j}
-\varphi_{k}\gamma^{k}_{\ ; ij}+2\varphi_j \varphi_k \gamma^{k}_{\ ; i}\bigg)\\ \nonumber&&-\frac{2\beta}{n} \dot{\varphi}
\frac{1}{1+|D \varphi|^2}\varphi^k\varphi_{l} \gamma^{l}_{\ ;k}
+(\alpha-1)\dot{\varphi}\varphi_l\gamma^{l},
\end{eqnarray*}
In view of
\begin{eqnarray*}
w^{ij}\varphi_{lj}=w^{ij}\sigma_{lj}-\delta^{i}_{l}+w^{ij}\varphi_{j}\varphi_{l},
\end{eqnarray*}
we obtain by the $C^1$-estimate \eqref{Gra-est}
\begin{eqnarray*}
\mathcal{L}w^{\prime}_{11}&\leq&c \ \dot{\varphi}(\ \mbox{tr}w^{ij}+1).
\end{eqnarray*}
Thus,
\begin{equation*}\begin{aligned}
\mathcal{L}W\leq&
\frac{\dot{\varphi}}{w_{11}+w^{\prime}_{11}+C}\bigg(
c \ \mbox{tr}w^{ij}+cw_{11}+c+
\frac{\alpha}{n}\dot{\varphi} w^{kl}_{\ \
;1}w_{kl;1}\bigg)
\\&+\frac{\dot{\varphi}}{w_{11}+w^{\prime}_{11}+C}\bigg(
c \ \mbox{tr}w^{ij}+c\bigg)
+\frac{\alpha}{n}\dot{\varphi}w^{ij}
\frac{(w_{11;i}+w^{\prime}_{11;i})(w_{11;j}+w^{\prime}_{11;j})}
{(w_{11}+w^{\prime}_{11}+C)^2}\\&-\frac{\alpha\lambda\dot{\varphi}}{n}\bigg(
(1+|D\varphi|^2)w^{ij}\varphi_{i}\varphi_j-|D\varphi|^2\bigg)\\
&-\frac{\alpha\lambda\dot{\varphi}}{n}\mbox{tr}w^{ij}+c\lambda\dot{\varphi}
+c(1-\lambda)\dot{\varphi}w_{11}.\end{aligned}
\end{equation*}
The last term which we have to estimate is
\begin{eqnarray*}
\frac{\alpha}{n}\dot{\varphi}\bigg(\frac{1}{w_{11}+w^{\prime}_{11}+C}
w^{kl}_{\ \
;1}w_{kl;1}+w^{ij}
\frac{(w_{11;i}+w^{\prime}_{11;i})(w_{11;j}+w^{\prime}_{11;j})}
{(w_{11}+w^{\prime}_{11}+C)^2}\bigg).
\end{eqnarray*}
For convenience later, we set $V=w_{11}+w^{\prime}_{11}+C$. Then,
\begin{equation*}\begin{aligned}
&\frac{1}{V}w^{kl}_{\ \
;1}w_{kl;1}+w^{ij}\frac{V_{i}V_{j}}
{V^2}\\=&-\frac{1}{V}w^{pk}w^{ql}w_{pq;1}w_{kl;1}+w^{ij}\frac{V_{i}V_{j}}{V^2}
\\ \leq&-\frac{1}{V}\frac{1}{w_{11}}w^{kl}w_{1k;1}w_{1l;1}+w^{ij}\frac{V_{i}V_{j}}{V^2}
\\=&w^{ij}\frac{V_{i}V_{j}}{Vw_{11}}-\frac{1}{V}\frac{1}{w_{11}}w^{kl}w_{1k;1}w_{1l;1}
-\frac{w^{\prime}_{11}+C}{V^2w_{11}}w^{ij}V_{i}V_{j}.\end{aligned}
\end{equation*}
In view of \eqref{C0},
together with the fact that the matric $w^{ij}$ is positive definite, we can say
\begin{eqnarray*}
-\frac{w^{\prime}_{11}+C}{V^2w_{11}}w^{ij}V_{i}V_{j}\leq 0.
\end{eqnarray*}
Thus,
\begin{eqnarray*}
\frac{1}{V}w^{kl}_{\ \
;1}w_{kl;1}+w^{ij}\frac{V{i}V_{j}}
{V^2}\leq \frac{1}{Vw_{11}}(w^{ij}V_iV_j-w^{kl}w_{1k;1}w_{1l;1}).
\end{eqnarray*}
Recalling that
\begin{eqnarray*}
w^{kl}V_kV_l=w^{kl}(w_{11;k}w_{11;l}+2w_{11;k}w^{\prime}_{11;l}
+w^{\prime}_{11;k}w^{\prime}_{11;l}).
\end{eqnarray*}
It follows from the equality \eqref{eq2}
\begin{equation}\label{last-1}\begin{aligned}
&\frac{1}{V}w^{kl}_{\ \
;1}w_{kl;1}+w^{ij}\frac{V_{i}V_{j}}
{V^2}\\ \leq &\frac{1}{Vw_{11}}(2w^{kl}w_{11;k}\varphi_lw_{11}-2
w_{11;1}\varphi_1-(w_{11})^2w^{kl}\varphi_k\varphi_l+w_{11}(\varphi_{1})^2
\\ &+2w^{kl}w_{11;k}w^{\prime}_{11;l}
+w^{kl}w^{\prime}_{11;k}w^{\prime}_{11;l}).\end{aligned}
\end{equation}
Since $W(x, t)$ achieves its maximum at
$(x_0, t_0)\in \Omega\times [0, T]$, so $W_i=0$ implies
\begin{eqnarray*}
W_i=\frac{V_i}{V}+\lambda\sigma^{kl}\varphi_{ki}\varphi_l=0.
\end{eqnarray*}
Therefore,
\begin{eqnarray}\label{11-1}
w_{11;1}=(-\lambda V\sigma^{kl}\varphi_{k1}\varphi_l-w^{\prime}_{11;1})
\end{eqnarray}
and
\begin{equation}\label{11-2}\begin{aligned}
w^{kl}w_{11;k}=&w^{kl}(-\lambda V\sigma^{pq}\varphi_{pk}\varphi_q-w^{\prime}_{11;k})
\\=&-\lambda Vw^{kl}\varphi_{pk}\sigma^{pq}\varphi_q-w^{kl}w^{\prime}_{11;k}
\\=&-\lambda V(w^{kl}\sigma_{pk}-\delta^{l}_{p}+w^{kl}\varphi_p \varphi_k)\sigma^{pq}\varphi_q-w^{kl}w^{\prime}_{11;k}
\\=&-\lambda Vw^{kl}\varphi_{k}(1+|D\varphi|^2)+\lambda V\sigma^{lp}\varphi_{p}
-w^{kl}w^{\prime}_{11;k}.\end{aligned}
\end{equation}
Then, we have by the $C^1$-estimate \eqref{Gra-est}, \eqref{ij}, \eqref{11-1} and \eqref{11-2}
\begin{equation*}\begin{aligned}
-2w_{11;1}\varphi_1&=-\lambda V\sigma^{kl}\varphi_{k1}\varphi_l\varphi_1
-w^{\prime}_{11;1}\varphi_1\\&=\lambda V\sigma^{kl}w_{k1}\varphi_l \varphi_1
-\lambda V (\varphi_1)^2-\lambda V |D \varphi|^2 (\varphi_1)^2+(\gamma^i \varphi_i)_1 \varphi_1
\\&=\lambda V\sigma^{kl}w_{k1}\varphi_l \varphi_1
-\lambda V (\varphi_1)^2-\lambda V |D \varphi|^2 (\varphi_1)^2\\&
+\gamma^{i}_{\ ; 1} \varphi_i\varphi_1-\gamma^i w_{i1}\varphi_1+\gamma^1\varphi_1
+\gamma^i\varphi_i(\varphi_1)^2\\&\leq c\lambda V(w_{11}+1)+c(w_{11}+1)\end{aligned}
\end{equation*}
and
\begin{equation*}\begin{aligned}
&\frac{1}{Vw_{11}}(2w^{kl}w_{11;k}\varphi_lw_{11}-(w_{11})^2w^{kl}\varphi_k\varphi_l+w_{11}(\varphi_{1})^2)\\ \leq&
-2\lambda(1+|D\varphi|^2)w^{kl}\varphi_k\varphi_l-\frac{2}{V}w^{kl}\varphi_lw^{\prime}_{11;k}
+c\lambda+\frac{c}{V}
\end{aligned}
\end{equation*}
Thus, combing the above inequalities and assume $w_{11}\geq 1$
(otherwise $w_{11}$ is upper bounded and our theorem holds true), we have
\begin{equation*}\begin{aligned}
&\frac{1}{Vw_{11}}(2w^{kl}w_{11;k}\varphi_lw_{11}-2
w_{11;1}\varphi_1-(w_{11})^2w^{kl}\varphi_k\varphi_l+w_{11}(\varphi_{1})^2)\\ \leq&
-2\lambda(1+|D\varphi|^2)w^{kl}\varphi_k\varphi_l-\frac{2}{V}w^{kl}\varphi_lw^{\prime}_{11;k}
+c\lambda+\frac{c}{V}+c\lambda\frac{w_{11}+1}{w_{11}}+\frac{c(w_{11}+1)}{V w_{11}}\\ \leq&
-2\lambda(1+|D\varphi|^2)w^{kl}\varphi_k\varphi_l+\frac{c}{V}(\mbox{tr}w^{kl}+1)
+c\lambda+\frac{c}{V}\end{aligned}
\end{equation*}
where we use the following inequality to get the last line,
\begin{equation}\begin{aligned}\label{cx2018}
w^{kl}\varphi_k w^{\prime}_{11; l}
&=w^{kl}\varphi_k (\gamma^{i}\varphi_{il}+\gamma^{i}_{\ ; l}\varphi_i)
\\&=(w^{kl}\sigma_{il}-\delta^{k}_{i}+w^{kl}\varphi_i \varphi_l)
\varphi_k \gamma^{i}+ w^{kl} \varphi_k \varphi_i \gamma^{i}_{\ ; l}
\\&\leq c(\mbox{tr} w^{kl}+1).\end{aligned}
\end{equation}
Now, we estimate the third line of \eqref{last-1}. Using  \eqref{11-2} and
assume $\lambda\geq 1$ and $w_{11}\geq 1$ (otherwise $w_{11}$ is upper bounded and our theorem holds true), we have
\begin{equation*}\begin{aligned}
&\frac{1}{Vw_{11}}(2w^{kl}w_{11;k}w^{\prime}_{11;l}
+w^{kl}w^{\prime}_{11;k}w^{\prime}_{11;l})\\=&
\frac{1}{Vw_{11}}(-2\lambda Vw^{kl}\varphi_k (1+|D\varphi|^2)w^{\prime}_{11;l}
+2\lambda V\sigma^{lp}\varphi_pw^{\prime}_{11;l}
-w^{kl}w^{\prime}_{11;k}w^{\prime}_{11;l})\\\leq &
\frac{1}{Vw_{11}}(\lambda V c \ (\mbox{tr} w^{kl}+1)
+c\lambda V(w_{11}+1))\\ \leq &
c\lambda (\frac{\mbox{tr} w^{kl}}{w_{11}}
+1)\end{aligned}
\end{equation*}
in view of \eqref{cx2018}, and
\begin{equation*}\begin{aligned}
\sigma^{lp}\varphi_p w^{\prime}_{11; l}
&=\sigma^{lp}\varphi_p(\gamma^{i}\varphi_{il}+\gamma^{i}_{\ ; l}\varphi_i)
\\&\leq c(w_{11}+1),\end{aligned}
\end{equation*}

\begin{equation*}\begin{aligned}
-w^{kl}w^{\prime}_{11;k}w^{\prime}_{11;l}\leq0.\end{aligned}
\end{equation*}
Inserting the above equality and \eqref{11-1} into \eqref{last-1}, we get at $(x_0, t_0)$
\begin{equation}\label{last-2}\begin{aligned}
&\frac{1}{V}w^{kl}_{\ \
;1}w_{kl;1}+w^{ij}\frac{V_{i}V_{j}}
{V^2}\\ \leq&
-2\lambda(1+|D\varphi|^2)w^{kl}\varphi_k\varphi_l+\frac{c}{V}(\mbox{tr}w^{kl}+1)
+c\lambda\bigg(\frac{\mbox{tr}w^{kl}}{w_{11}}
+1\bigg).\end{aligned}
\end{equation}
Thus,
\begin{equation*}\begin{aligned}
\mathcal{L}W \leq&
\frac{\dot{\varphi}}{V}\bigg(c\
\mbox{tr}w^{ij}+cw_{11}+c\bigg)
+\frac{\dot{\varphi}}{V}\bigg(
c\ \mbox{tr}w^{ij}+c\bigg)
\\&-\frac{\alpha\lambda\dot{\varphi}}{n}\bigg(
(1+|D\varphi|^2)w^{ij}\varphi_{i}\varphi_j-2|D\varphi|^2\bigg)\\
&-\frac{\alpha\lambda\dot{\varphi}}{n}\mbox{tr}w^{ij}+c\lambda\dot{\varphi}
+c\dot{\varphi}(1-\lambda) w_{11}\\ &
+\frac{\alpha\dot{\varphi}}{n}\bigg(-2\lambda(1+|D\varphi|^2)w^{kl}\varphi_k\varphi_l
+\frac{c}{V}(\mbox{tr}w^{kl}+1)+c\lambda(\frac{\mbox{tr}w^{kl}}{w_{11}}+1)\bigg)\\ \leq&
\frac{\dot{\varphi}}{V}c\bigg(
\mbox{tr}w^{ij}+w_{11}+1\bigg)-\frac{\alpha\lambda\dot{\varphi}}{n}\bigg(
3(1+|D\varphi|^2)w^{ij}\varphi_{i}\varphi_j-2|D\varphi|^2\bigg)\\
&-\frac{\alpha\lambda\dot{\varphi}}{n}\mbox{tr}w^{ij}+c\lambda\dot{\varphi}
+c \ \dot{\varphi}(1-\lambda) w_{11}\\ &
+\frac{\alpha\dot{\varphi}}{n}c\lambda(\frac{\mbox{tr}w^{kl}}{w_{11}}+1)\\ \leq&
\frac{\dot{\varphi}}{V}c\bigg(
\mbox{tr}w^{kl}+w_{11}+1\bigg)
\\&-\frac{\alpha\lambda\dot{\varphi}}{n}\mbox{tr}w^{kl}+c\lambda\dot{\varphi}
+c\dot{\varphi}(1-\lambda)w_{11}\\ &
+\frac{\alpha\dot{\varphi}}{n}c\lambda(\frac{\mbox{tr}w^{kl}}{w_{11}}+1)
\\ \leq&\dot{\varphi}\mbox{tr}w^{kl}(\frac{c\lambda}{w_{11}}+c-\frac{\alpha\lambda}{n})
+c\dot{\varphi}\bigg(1+\lambda+(1-\lambda)w_{11}\bigg).\end{aligned}
\end{equation*}
Since $\dot{\varphi} >0$, we take $\lambda$ and $w_{11}$ are large enough such that $(\frac{c\lambda}{w_{11}}+c-c\lambda)\leq 0$, otherwise $w_{11}$ is upper bounded. In view of $\mathcal{L}W\geq0$, we obtain
\begin{eqnarray*}
w_{11}\leq c,
\end{eqnarray*}
we conclude that $w_{11}$ has upper bounded which depends on $\alpha$. Thus, the second covariant
derivatives of $\varphi$ is bounded from below.
\end{proof}

\subsection{Double normal $C^2$ boundary estimates}

\

Let
\begin{equation*}\begin{aligned}
\mathcal{\widetilde{L}}U=&\dot{U}-Q^{ij}U_{ij}-\frac{2\beta}{n}
\frac{\dot{\varphi}}{1+|D\varphi|^2}\varphi^kU_k\\=&
\dot{U}-\frac{\alpha}{n}\dot{\varphi}w^{ij}U_{ij}-\frac{2\beta}{n}
\frac{\dot{\varphi}}{1+|D\varphi|^2}\varphi^kU_k\end{aligned}
\end{equation*}
and
\begin{eqnarray*}
q(x)=-d(x)+\eta d^2(x),
\end{eqnarray*}
where $d$ denotes the distance to $\partial\Omega$ which is a smooth function in
$\Omega_{\delta}=\{x\in \Omega: dist(x, \partial\Omega)<\delta\}$ for $\delta$ small enough
and $\eta$ denotes a constant to be chosen sufficiently large. Thus,
$q: \Omega_{\delta}\rightarrow \mathbb{R}$ is a smooth function.

To derive double normal
$C^2$ boundary estimates, we need the following lemma.

\begin{lemma}\label{q}
For any solution  $\varphi$ of the flow \eqref{Evo-1},
we can choose $\eta$ so large and $\delta$ so small such that
\begin{eqnarray*}
\mathcal{\widetilde{L}} q(x)\leq -\frac{\alpha}{4n} k_0\dot{\varphi}\ \mbox{\tr}(w^{ij})
\quad \mbox{in} \quad \Omega_{\delta},
\end{eqnarray*}
where $k_0$ is a positive
constant depending on $\partial\Omega$.
\end{lemma}

\begin{proof}
Differentiating the function $q$ twice with $x$,
\begin{eqnarray}\label{q-1}
q_i(x)=-d_i(x)+2\eta d(x)d_{i}(x)
\end{eqnarray}
and
\begin{eqnarray}\label{q-2}
q_{ij}(x)=-d_{ij}(x)+2\eta d_{i}(x)d_{j}(x)+2\eta d(x)d_{ij}(x).
\end{eqnarray}
For any $x_0\in \partial \Omega$, after a rotation of
the first $n-1$ coordinates and remembering that $\mu(x_0)=e_n$, we have
\begin{equation*}
d_{ij}(x_0)=\left(
              \begin{array}{cccc}
                 -\kappa_1 & 0 & \cdot\cdot\cdot & 0 \\
                 \cdot\cdot\cdot& \cdot\cdot\cdot & \cdot\cdot\cdot & \cdot\cdot\cdot \\
                0 & 0 & \cdot\cdot\cdot & -\kappa_{n-1} \\
                0& 0 & \cdot\cdot\cdot & 0 \\
              \end{array}
            \right),
\end{equation*}
where there is a constant $k_0=k_0(\partial\Omega)>0$ such that $\kappa_i\geq k_0$ for all
principle curvature $\kappa_i, i=1,2,...,n-1$ of $\partial\Omega$ and for any $x_0\in \partial \Omega$.
Since the differential of the distance coincide withe the inward normal vector $-Dd(x_0)=\mu(x_0)=e_n$.
Thus, it holds at $x_0$
\begin{eqnarray*}
q_{ij}(x_0)=\left(
              \begin{array}{cccc}
                 \kappa_1(1-2\eta d) & 0 & \cdot\cdot\cdot & 0 \\
                 \cdot\cdot\cdot& \cdot\cdot\cdot & \cdot\cdot\cdot & \cdot\cdot\cdot \\
                0 & 0 & \cdot\cdot\cdot & \kappa_{n-1}(1-2\eta d) \\
                0& 0 & \cdot\cdot\cdot & 2\eta \\
              \end{array}
            \right).
\end{eqnarray*}
Choosing $\eta \delta\leq\frac{1}{4}$, we have
\begin{eqnarray*}
w^{ij}q_{ij}\geq\frac{1}{2}k_0(w^{11}+w^{22}
+...+w^{n-1 \ n-1})+2\eta w^{nn}.
\end{eqnarray*}
On the one hand, we can choose $\eta\geq\frac{1}{4}k_0$ such that
\begin{eqnarray}\label{D22}
w^{ij}q_{ij}\geq \frac{1}{2}k_0\mbox{tr}(w^{ij}).
\end{eqnarray}
On the other hand, using the inequality of arithmetic and geometric, we obtain
\begin{eqnarray*}
w^{ij}q_{ij}\geq c(n, k_0)\eta^{\frac{1}{n}}
\bigg(\prod_{i=1}^{n}w^{ii}\bigg)^{\frac{1}{n}}.
\end{eqnarray*}
The Hadamard'inequality \cite{Hor} for positive definite matrices
\begin{eqnarray*}
\det(w^{ij})\leq
\bigg(\prod_{i=1}^{n}w^{ii}\bigg)
\end{eqnarray*}
implies
\begin{eqnarray*}
w^{ij}q_{ij}\geq c(n, k_0)\eta^{\frac{1}{n}}
\mbox{det}(w^{ij})^{\frac{1}{n}}.
\end{eqnarray*}
Recalling \eqref{w-ij}, there is a positive constant $c_2$ such that
\begin{eqnarray*}
\mbox{det}(w^{ij})=\mbox{det}^{-1}(w_{ij})\geq\frac{1}{c_2}>0,
\end{eqnarray*}
it follows that
\begin{eqnarray*}\label{key estimate}
w^{ij}q_{ij}\geq \frac{1}{c_2}c(n, k_0)\eta^{\frac{1}{n}}.
\end{eqnarray*}
Using the $C^1$-estimate \eqref{Gra-est}, we have
\begin{eqnarray*}
\bigg|\frac{2\beta}{n}\frac{1}{1+|D\varphi|^2}\varphi^kq_k\bigg|=\bigg|\frac{2\beta}{n}\frac{1}{1+|D\varphi|^2}
\varphi^k(-d_k+2\eta dd_{k})\bigg|\leq c_3(1+\eta \delta),
\end{eqnarray*}
for all $(x, t)\in \Omega_{\delta}\times [0, T]$.
Choose $\eta$ so large and $\delta$ so small such that
\begin{eqnarray*}
\frac{1}{2}\frac{1}{c_2}c(n, k_0)\eta^{\frac{1}{n}}\geq c_3(1+\eta \delta).
\end{eqnarray*}
Thus, we have from \eqref{D22}
\begin{eqnarray*}
\mathcal{\widetilde{L}} q(x)&=&
-\frac{\alpha}{2n}\dot{\varphi}w^{ij}q_{ij}-\frac{2\beta}{n}\dot{\varphi}
\frac{1}{1+|D \varphi|^2}\varphi^k q_k-\frac{\alpha}{2n}\dot{\varphi}w^{ij}q_{ij}\\&\leq&-\frac{1}{2}\frac{1}{c_2}c(n, k_0)\eta^{\frac{1}{n}}+c_3(1+\eta \delta)
-\frac{\alpha}{2n}\dot{\varphi}w^{ij}q_{ij}\\&\leq&-\frac{\alpha}{4n} k_0\dot{\varphi}\mbox{tr}(w^{ij}).
\end{eqnarray*}
\end{proof}

Clearly, choosing $\frac{1}{8}\leq \eta \delta\leq \frac{1}{4}$,
from \eqref{q-1} and \eqref{q-2}, we make sure that
$q$ satisfied the following properties in $\Omega_{\delta}$:
\begin{eqnarray*}
-\delta+\eta\delta^{2}\leq q(x) \leq 0,
\end{eqnarray*}
\begin{eqnarray}\label{Dq-1}
\frac{1}{2}\leq |D q|\leq 1,
\end{eqnarray}
\begin{eqnarray}\label{Dq-2}
\frac{k_{0}}{2}\sigma_{ij}\leq D^2 q\leq C(\partial \Omega)(1+\eta)\sigma_{ij}
\end{eqnarray}
and
\begin{eqnarray}\label{Dq-3}
|D^3 q|\leq C(\partial \Omega)(1+\eta).
\end{eqnarray}

It is easy to see
\begin{eqnarray}\label{nor}
\frac{D q}{|D q|}=\mu
\end{eqnarray}
for unit outer normal $\mu$ on the boundary $\partial\Omega$.
We consider the following function
$$P(x, t)=D\varphi \cdot Dq+A q(x),$$
where the constant $A$ will be choose later.

\begin{lemma}
For any solution  $\varphi$ of the flow \eqref{Evo-1} in
$\Omega\times [0, T]$ for some fixed $T<T^{*}$, we have
$$\mathcal{\widetilde{L}}P(x, t)\leq 0.$$
\end{lemma}

\begin{proof}
The calculation of $\mathcal{\widetilde{L}}P(x, t)$ is similar to that of \eqref{w11-}.
We derivative this function $P(x, t)$ twice with $x$
\begin{eqnarray*}
P_{i}=\varphi_{li}q^{l}+\varphi_{l}q^{l}_{\ i}+Aq_i
\end{eqnarray*}
and
\begin{eqnarray*}
P_{ij}=\varphi_{lij}q^{l}+\varphi_{li}q^{l}_{\ j}+\varphi_{lj}q^{l}_{\ i}+\varphi^{l}q_{lij}+Aq_{ij}.
\end{eqnarray*}
Differentiating $P(x, t)$  with $t$, we have
\begin{equation*}\begin{aligned}
P_{t}=&D \varphi_{t}\cdot D q\\=&\bigg(Q^{ij}\varphi_{ijl}+Q^{k}\varphi_{kl}
+(\alpha-1)Q\varphi_l\bigg)q^{l}\\
=&\bigg(Q^{ij}\varphi_{lij}+Q^{ij}\varphi_{j}\sigma_{il}
-Q^{ij}\varphi_{l}\sigma_{ij}+Q^{k}\varphi_{kl}+(\alpha-1)Q\varphi_l\bigg)q^{l}.\end{aligned}
\end{equation*}
Therefore, we have
\begin{equation*}\begin{aligned}
\mathcal{\widetilde{L}}P(x, t)=&P_t-Q^{ij}P_{ij}-2\beta\frac{\dot{\varphi}}{1+|D\varphi|^2}\varphi^kP_k
\\=&-2Q^{ij}\varphi_{li}q^{l}_{\ j}
+ Q^{ij}\bigg(\sigma_{il}\varphi_jq^l
-\sigma_{ij}\varphi_l q^l\bigg)\\&- Q^{ij}\varphi^{l}q_{lij}-2\beta\frac{\dot{\varphi}}{1+|D\varphi|^2}\varphi^k\varphi^l q_{lk}\\&
-\frac{2\alpha}{n}\dot{\varphi} w^{kl}\varphi_l\varphi_{km}q^m+(\alpha-1)Q\varphi_lq^{l}+A\mathcal{L}q(x).\end{aligned}
\end{equation*}
Since
\begin{eqnarray*}
w^{ij}\varphi_{lj}=w^{ij}\sigma_{lj}-\delta^{i}_{l}+w^{ij}\varphi_{j}\varphi_{l},
\end{eqnarray*}
we obtain by using the $C^1$-estimate \eqref{Gra-est} and \eqref{Dq-1}, \eqref{Dq-2}, \eqref{Dq-3}
\begin{eqnarray*}
\mathcal{\widetilde{L}}P(x, t)&\leq&C(1+\eta)\dot{\varphi}\mbox{tr}w^{ij}+C(1+\eta)\dot{\varphi}+A\mathcal{L}q(x).
\end{eqnarray*}
Using Lemma \ref{q}, we get
\begin{eqnarray*}
\mathcal{\widetilde{L}}P(x, t)\leq C\dot{\varphi}\bigg((1+\eta-A)\mbox{tr}w^{ij}+(1+\eta)\bigg).
\end{eqnarray*}
Recalling \eqref{w-ij},
\begin{eqnarray*}
\bigg(\frac{\mbox{tr}(w^{ij})}{n}\bigg)^n\geq\mbox{det}(w^{ij})=\mbox{det}^{-1}(w_{ij})\geq\frac{1}{c_2}>0.
\end{eqnarray*}
Choosing $A\geq \frac{c_{2}^{\frac{1}{n}}}{n}(1+\eta)+\eta+1$, we get
\begin{eqnarray*}
\mathcal{\widetilde{L}}P(x, t)\leq 0.
\end{eqnarray*}
\end{proof}

\begin{proposition}\label{NN}
For $\varphi$ be a solution of the flow \eqref{Evo-1} in
$\Omega\times [0, T]$ for some fixed $T<T^{*}$, $\varphi_{\mu\mu}$ is
uniformly bounded from below, i.e., there exists $C=C(n, M_0)$
such that
\begin{equation*}
-\varphi_{\mu\mu}\leq C(n, M_0),
\qquad \ \ \forall (x, t)\ \ \mbox{on} \ \  \partial\Omega\times [0, T],
\end{equation*}
where $\varphi_{\mu\mu}:=\varphi_{ij}\mu^i\mu^j$.
\end{proposition}

\begin{proof}
It is easy to see from the boundary condition in the flow \eqref{Evo-1}
\begin{eqnarray*}
P=0 \ \ \mbox{on} \ \ \partial \Omega\times [0, T].
\end{eqnarray*}
On the $(\partial \Omega_{\delta}\setminus \partial \Omega)\times [0, T]$, we have
\begin{eqnarray*}
P\leq C-A\delta\leq 0,
\end{eqnarray*}
provided $A\geq \frac{C}{\delta}$. Applying the maximum principle, it follows that
\begin{eqnarray*}
P\leq0 \ \ \mbox{in} \ \ \Omega_{\delta}\times [0, T].
\end{eqnarray*}
Assume $(x_0, t_0)\in \partial\Omega\times [0,T]$ is the minimum point
of $\varphi_{\mu\mu}$ on $\partial\Omega\times [0,T]$, using the $C^1$-estimate \eqref{Gra-est},
we have by noticing \eqref{nor}
\begin{eqnarray*}
0\leq P_{\mu}(x_0, t_0)=\varphi_{i\mu}q^i+\varphi^i q_{i\mu}+A q_{\mu}\leq \varphi_{\mu\mu}+C+A.
\end{eqnarray*}
Therefore,
\begin{eqnarray*}
-\varphi_{\mu\mu}\leq C+A.
\end{eqnarray*}
\end{proof}

\subsection{Remaining $C^2$ boundary estimates}

\

We have obtained interior estimates under the assumption that the maximum of $W$ is in the
interior of $\Omega$. Now we have to contemplate
the possibility that the maximum of $W$ is not in the
interior of $\Omega$.
Since the double normal boundary estimates have been done in the previous subsection,
we shall follow the similar discussion as those done by
Lions-Trudinger-Urbas in \cite{LTU} to get remaining $C^2$ boundary estimates.

\begin{proposition}\label{4.071}
Let $\varphi$ be a solution of the flow \eqref{Evo-1} in
$\Omega\times [0, T]$ for some fixed $T<T^{*}$, assume $W$ attains
its maximum on $\partial \Omega\times \mathbb{S}^{n-1}\times
[0, T]$. Then, there exists $C=C(n, M_0)$ such that
\begin{equation*}
C(n, M_0)\leq \varphi_{ij}(x, t)\xi^i\xi^j, \qquad \ \ \forall (x,
\xi, t)\in \partial\Omega\times \mathbb{S}^{n-1}\times [0, T].
\end{equation*}
\end{proposition}

\begin{proof}
Assume $W$ attains its maximum at a point $(x_0, \xi_0, t_0)\in
\partial\Omega\times S^{n-1}\times [0, T]$. From Proposition
\ref{NN},we know
\begin{equation*}
C(n, M_0)\leq \varphi_{\mu\mu}, \qquad \ \ \forall ( x, t)\in
\partial\Omega\times[0, T].
\end{equation*}
Thus, the remaining case is $\xi_0\neq \mu$. Without loss of
generality that $W$ attains its maximum at a point $(x_0, \xi_0,
t_0)\in \partial\Omega\times S^{n-1}\times [0, T]$ with $\xi_0\neq
\mu$. We represent $\partial \Omega$ locally as graph $f$ over its
tangent plane at a fixed point $x_0\in \partial \Omega$ such that
$\Omega=\{(x^n, \widehat{x}): x^n<f(\widehat{x})\}$ and we
distinguish two cases.

(i) $\xi_0$ is tangential: if $\xi_0$ is tangential to $\partial \Omega$, we differentiate the boundary condition
\begin{eqnarray*}
\mu^i \varphi_i=0
\end{eqnarray*}
with respect to tangential directions $\xi_0$
\begin{eqnarray*}
\mu^{i}_{\ ;\xi_0} \varphi_i+\mu^i\varphi_{i\xi_0}
+\mu^i\varphi_{in}f_{\xi_0}=0,
\end{eqnarray*}
then at $x_0$
\begin{eqnarray*}
\mu^i_{\ ;\xi_0} \varphi_i+\mu^i\varphi_{i\xi_0}=0
\end{eqnarray*}
in view of $Df(\widehat{x}_0)=0$, which implies together with the $C^1$-estimate \eqref{Gra-est}
\begin{eqnarray*}
|\mu^i\varphi_{i\xi_0}|\leq c.
\end{eqnarray*}
We differentiate the boundary condition again
and we get at $x_0$ in view of $Df(\widehat{x}_0)=0$
\begin{eqnarray*}
\mu^i_{\ ;\xi_0\xi_0} \varphi_i+2\mu^i_{\ ;\xi_0} \varphi_{i\xi_0}+\mu^i \varphi_{i\xi_0\xi_0}
+\mu^i\varphi_{in}f_{\xi_0\xi_0}=0.
\end{eqnarray*}

\begin{remark}
We put vectors as indices to indicate products as
$$\varphi_{\mu\xi\xi}:=\mu^{i}\varphi_{ijk}\xi^{j}\xi^{k}$$
and not covariant derivatives in the corresponding direction
$$\varphi_{i\xi\xi}\neq(\xi^{j}\varphi_{ij})_{; \xi^{k}}=\xi^{k}\xi^{j}_{; k}\varphi_{ij}+\xi^{k}\xi^{j}\varphi_{ijk}.$$
Analogously we write
$$\varphi_{\xi\xi}:=\varphi_{ij}\xi^{i}\xi^{j}, \quad w_{\xi\xi}:=w_{ij}\xi^{i}\xi^{j}, \quad w_{\mu\mu}:=w_{ij}\mu^{i}\mu^{j},
\quad \sigma_{\xi\xi}:=\sigma_{ij}\xi^{i}\xi^{j},...$$
\end{remark}

$C^1$-estimates \eqref{Gra-est} and double normal estimates provide at $x_0$
$$\mu^i_{\ ;\xi_0\xi_0} \varphi_i\leq c$$
and
$$\mu^i\varphi_{in}f_{\xi_0\xi_0}\leq c$$
in view of $D^2f(\widehat{x}_0)<0$.
So we obtain
\begin{eqnarray}\label{Tan1}
\varphi_{\mu \xi_0\xi_0}\geq-2\mu^i_{\ ;\xi_0} \varphi_{i\xi_0}-c
\geq2\mu^i_{\ ;\xi_0} (1+\varphi_{i}\varphi_{\xi_0}-\varphi_{i\xi_0})-
c=2\mu^i_{\ ;\xi_0}w_{i\xi_0}-
c,
\end{eqnarray}
As already noted, $\xi_0$ is an eigenvector of
$w_{ij}(x_0, t_0)+w^{\prime}_{ij}(x_0, t_0)$ to
an eigenvalue $\lambda_0$, since it corresponds to a maximal direction. Therefore, it holds
\begin{eqnarray*}
\mu^i_{\ ;\xi_0}w_{i\xi_0}(x_0, t_0)
&=&\xi_{0}^{j}\mu^i_{\ ; j}(w_{ik}
+w^{\prime}_{ik})\xi_{0}^{k}-\xi_{0}^{j}\mu^i_{\ ; j}w^{\prime}_{ik}\xi_{0}^{k}
\\&=&\lambda_0\xi_{0}^{j}\mu^i_{\ ; j}\sigma_{ik}
\xi_{0}^{k}-\xi_{0}^{j}\mu^i_{\ ; j}w^{\prime}_{ik}\xi_{0}^{k},
\end{eqnarray*}
where we may assume that $\lambda_0$ is nonnegative, because otherwise $w_{ik}
+w^{\prime}_{ik}$ would be negative definite and the needed estimate would follow
immediately. Moreover the strict convexity of $\partial\Omega$ implies the existence of a
constant $c_1>0$ such that
\begin{eqnarray*}
\xi^{j}\mu^i_{\ ; j}\sigma_{ik}\xi^{k}\geq c_1\xi^{i}\sigma_{ik}\xi^{k}
\end{eqnarray*}
for all tangential vectors $\xi$. Thus, the inequality \eqref{Tan1} becomes
\begin{eqnarray*}
\varphi_{\mu \xi_0\xi_0}&\geq&2\lambda_0\xi_{0}^{j}\mu^i_{\ ; j}\sigma_{ik}
\xi_{0}^{k}-2\xi_{0}^{j}\mu^i_{\ ; j}w^{\prime}_{ik}\xi_{0}^{k}-c\\&\geq&
2c_1\lambda_0\xi_{0}^{i}\sigma_{ik}
\xi_{0}^{k}-2\xi_{0}^{j}\mu^i_{\ ; j}w^{\prime}_{ik}\xi_{0}^{k}-c
\\&=&2c_1\xi_{0}^{i}(w_{ik}
+w^{\prime}_{ik})\xi_{0}^{k}-2\xi_{0}^{j}\mu^i_{\ ; j}w^{\prime}_{ik}\xi_{0}^{k}-c
\\&\geq&2c_1w_{\xi_{0}\xi_{0}}-c.
\end{eqnarray*}
On the other hand, the
maximality of $W$ at $x_0$ gives $0\leq W_{\mu},$
\begin{eqnarray*}
0\leq \frac{w_{\xi_0 \xi_0; \mu}+w^{\prime}_{\xi_0 \xi_0; \mu}}{V}+\lambda\varphi^i\varphi_{i\mu}.
\end{eqnarray*}
Since  $\partial\Omega$ is strictly convex, by \eqref{010601}, we have
\begin{eqnarray*}
\lambda\varphi^i\varphi_{i\mu}
=-\lambda\varphi^i\mu_{\ ;i}^{j}\varphi_{j}\leq-c_1\lambda|D\varphi|^2\leq 0,
\end{eqnarray*}
which implies that
\begin{eqnarray*}
0\leq-\varphi_{\xi_0 \xi_0 \mu}+c,
\end{eqnarray*}
together with
\begin{eqnarray*}
-\varphi_{\xi_0\xi_0\mu}=-\varphi_{\mu\xi_0\xi_0}-R_{\mu \xi_0\xi_0 i}\varphi^i,
\end{eqnarray*}
thus
\begin{eqnarray*}
W(x_0, \xi_0, t_0)\leq c.
\end{eqnarray*}
So we obtain the desired estimate
\begin{equation*}
C(n, \Sigma_0)\leq \varphi_{ij}(x, t)\xi^i\xi^j, \qquad \ \ \forall
(x, \xi, t)\in \overline{\Omega}\times \mathbb{S}^{n-1}\times [0,
T].
\end{equation*}

(ii) $\xi_0$ is non-tangential: if $\xi_0$ is neither tangential nor normal we
need the tricky choice of \cite{LTU}. We find $0<\vartheta<1$ and a
tangential direction $\tau$ such that
\begin{eqnarray*}
\xi_0=\vartheta\tau+\sqrt{1-\vartheta^2}\mu.
\end{eqnarray*}
Thus,
\begin{eqnarray*}
\varphi_{\xi_0\xi_0}=\vartheta^2
\varphi_{\tau\tau}+(1-\vartheta^2)\varphi_{\mu\mu}
+2\vartheta\sqrt{1-\vartheta^2}\varphi_{\tau\mu}.
\end{eqnarray*}
Differentiating the boundary condition at a boundary point, we have
\begin{eqnarray*}
\mu^{i}_{\ ;j}\varphi_i=-\mu^{i}\varphi_{ij}.
\end{eqnarray*}
Therefore, at the boundary point
\begin{eqnarray*}
w^{\prime}(x, \xi_0, \xi_0, t)
=-2\mu^{i}_{\ ;j}\varphi_{i}\langle\xi_0,
\mu\rangle\xi_{0}^{\prime j}=2\vartheta\sqrt{1-\vartheta^2}\varphi_{\tau\mu},
\end{eqnarray*}
and consequently,
\begin{eqnarray*}
\varphi_{\xi_0\xi_0}=\vartheta^2
\varphi_{\tau\tau}+(1-\vartheta^2)\varphi_{\mu\mu}
+w^{\prime}_{\xi_0\xi_0}.
\end{eqnarray*}
Thus, in view of the Neumann boundary condition,
\begin{eqnarray*}
w_{\xi_0\xi_0}+w^{\prime}_{\xi_0\xi_0}=1+\vartheta^2\varphi_\tau
\varphi_\tau-(\vartheta^2\varphi_{\tau\tau}+(1-\vartheta^2)\varphi_{\mu\mu}),
\end{eqnarray*}
which means we can rewrite $\exp(W-\frac{1}{2}\lambda|D\varphi|^2)-C$ as
\begin{eqnarray*}
1+\vartheta^2\varphi_\tau \varphi_\tau-(\vartheta^2\varphi_{\tau\tau}+(1-\vartheta^2)\varphi_{\mu\mu}),
\end{eqnarray*}
so we obtain in view of the maximality of $W$ and the
fact that $\exp(W-\frac{1}{2}\lambda|D\varphi|^2)-C$ is independent of $\xi$
\begin{eqnarray*}
1+\varphi_\tau \varphi_\tau-\varphi_{\tau\tau}\leq1+\vartheta^2\varphi_\tau \varphi_\tau-(\vartheta^2\varphi_{\tau\tau}+(1-\vartheta^2)\varphi_{\mu\mu}),
\end{eqnarray*}
which implies
\begin{eqnarray*}
-\varphi_{\tau\tau}\leq-\varphi_{\mu\mu}.
\end{eqnarray*}
Therefore,
\begin{eqnarray*}
W(x, \tau, t)\leq W(x, \mu, t)+c.
\end{eqnarray*}
So
\begin{eqnarray*}
W(x, \xi_0, t)\leq c
\end{eqnarray*}
in view of Proposition \ref{NN}. Thus, we obtain the desired estimate
\begin{equation*}
C(n, \Sigma_0)\leq \varphi_{ij}(x, t)\xi^i\xi^j, \qquad \ \ \forall
(x, \xi, t)\in \partial\Omega\times \mathbb{S}^{n-1}\times [0,
T].
\end{equation*}
\end{proof}

\begin{theorem}
Under the hypothesis of Theorem \ref{main1.1}, we conclude
\begin{equation*}
T^{*}=+\infty.
\end{equation*}
\end{theorem}
\begin{proof}
Recalling that $\varphi$ satisfies the equation \eqref{Evo-1}
\begin{equation*}
\frac{\partial \varphi}{\partial t}=Q(x, \varphi, D\varphi, D^{2}\varphi).
\end{equation*}
By a simple calculation, we get
\begin{eqnarray*}
\frac{\partial Q}{\partial
\varphi_{ij}}=\frac{\alpha}{n}e^{(\alpha-1)\varphi}
(1+|D\varphi|^2)^{\frac{\beta}{n}}\frac{\det^{\frac{\alpha}{n}}(\sigma_{kl})}
{\det^{\frac{\alpha}{n}}(w_{kl})}w^{ij},
\end{eqnarray*}
which is uniformly parabolic on finite intervals from $C^0$-estimate
\eqref{C^0}, $C^1$-estimates \eqref{Gra-est} and the estimate
\eqref{w-ij}. Then by Evans-Krylov estimate \cite{Ev} \cite{Kr}, or the
results of chapter 14 in \cite{Lieb}, we have
\begin{equation*}
|\varphi|_{C^{2, \alpha}(\overline{\Omega})}\leq C(n, M_0, T^*),
\end{equation*}
which implies the maximal time interval is unbounded, i.e.,
$T^*=+\infty$.
\end{proof}

\section{Convergence of the rescaled flow}

\

Now, we define the  rescaled flow by
\begin{equation*}
\widetilde{X}=X\Theta^{-1}.
\end{equation*}
Thus,
\begin{equation*}
\widetilde{u}=u\Theta^{-1},
\end{equation*}
\begin{equation*}
\widetilde{\varphi}=\varphi-\log\Theta,
\end{equation*}
and the rescaled Gauss curvature
\begin{equation*}
\widetilde{K}=K\Theta^{n}.
\end{equation*}
Then, the rescaled scalar curvature equation takes the form
\begin{equation*}
\frac{\partial}{\partial t}\widetilde{u}=v\widetilde{K}^{-\frac{\alpha}{n}}\Theta^{\alpha-1}
-\widetilde{u}\Theta^{\alpha-1}.
\end{equation*}
Defining $s=s(t)$ by the relation
\begin{equation*}
\frac{ds}{d t}=\Theta^{\alpha-1}
\end{equation*}
such that $s(0)=0$ we conclude that $s$ ranges from $0$ to $+\infty$
and $\widetilde{u}$ satisfies
\begin{equation*}
\frac{\partial}{\partial
s}\widetilde{u}=v\widetilde{K}^{-\frac{\alpha}{n}} -\widetilde{u},
\end{equation*}
or equivalently, with $\widetilde{\varphi}=\log \widetilde{u}$
\begin{equation}\label{R-1}
\frac{\partial}{\partial
s}\widetilde{\varphi}=v\widetilde{u}^{-1}\widetilde{K}^{-\frac{\alpha}{n}}
-1=\widetilde{Q}(\widetilde{\varphi}, D\widetilde{\varphi},
D^2\widetilde{\varphi}).
\end{equation}
Since the spatial derivatives of $\widetilde{\varphi}$ are identical
to those of $\varphi$, \eqref{R-1} is a nonlinear parabolic equation
with a uniformly parabolic and concave operator $\widetilde{K}$.
Then, using the decay estimate estimate \eqref{Gra-est-10} of $|D\varphi|$,
we can deduce a decay
estimate of $|D\widetilde{\varphi}(\cdot, s)|$:
\begin{lemma}
Let $\varphi$ be a solution of \eqref{Evo-1}, then we have for
$0<\alpha<1$
\begin{equation}\label{Gra-est-}
|D\widetilde{\varphi}(x, s)|\leq \sup_{\overline{\Omega}}c\cdot e^{-(1-\alpha)c\cdot
s}|D\widetilde{\varphi}(\cdot, 0)|,
\end{equation}
where $c$ is a positive constant.
\end{lemma}
Thus, we can apply the Evans-Krylov theorem \cite{Ev} \cite{Kr} and
thereafter the parabolic Schauder estimate to conclude:
\begin{lemma}\label{rescaled flow}
Let $\varphi$ be a solution of the inverse Gauss curvature flow
\eqref{Evo-1}. Then,
\begin{equation*}
\widetilde{\varphi}(\cdot, s).
\end{equation*}
converges to a real number for $s\rightarrow +\infty$.
\end{lemma}

So, we have
\begin{theorem}\label{rescaled flow}
The rescaled flow
\begin{equation*}
\frac{d
\widetilde{X}}{ds}=\widetilde{K}^{-\frac{\alpha}{n}}\nu-\widetilde{X}
\end{equation*}
exists for all time and the leaves converge in $C^{\infty}$ to a
piece of round sphere.
\end{theorem}

\end{document}